\documentclass[12pt,reqno] {amsart}
\usepackage{amsfonts}
\usepackage{amsmath}
\usepackage{amssymb}
\usepackage{xypic,color}
\usepackage{graphicx}%
 \theoremstyle{plain}
\numberwithin{equation}{section}

\newtheorem{problem}{Problem}

\newtheorem{theorem}{Theorem}[section]
\newtheorem{conjecture}[theorem]{Conjecture}
\newtheorem{corollary}[theorem]{Corollary}
\newtheorem{definition}[theorem]{Definition}
\newtheorem{lemma}[theorem]{Lemma}
\newtheorem{proposition}[theorem]{Proposition}
\newtheorem{remark}[theorem]{Remark}
\newtheorem{example}[theorem]{Example}

\def\bn{\begin{definition}}
\def\en{\end{definition}}
\def\ba{\begin{array}}
\def\ea{\end{array}}
\def\be{\begin{equation}}
\def\ee{\end{equation}}
\def\bd{\begin{description}}
\def\ed{\end{description}}
\def\bu{\begin{enumerate}}
\def\eu{\end{enumerate}}
\def\bi{\begin{itemize}}
\def\ei{\end{itemize}}

\newcommand{\BBB}{\mathcal B}

\def\R{\mathbb{R}}
\def\N{\mathbb{N}}

\def\Re{\text{Re}}

\def\i{\mathfrak{i}}
\allowdisplaybreaks
\def\<{\langle}
\def\>{\rangle}
\usepackage[colorlinks,
            linkcolor=black,
            anchorcolor=black,
            citecolor=black,
            ]{hyperref}

\makeatletter
\newtheorem*{rep@theorem}{\rep@title}
\newcommand{\newreptheorem}[2]{%
\newenvironment{rep#1}[1]{%
 \def\rep@title{#2 \ref{##1} (restated)}%
 \begin{rep@theorem}}%
 {\end{rep@theorem}}}
\makeatother
 \newreptheorem{theorem}{Theorem}
  \newreptheorem{proposition}{Proposition}
 
\begin{document}
\title[]
{M\"{o}bius disjointness for a class of exponential functions}

\author[]{Weichen Gu}

\address{Department of Mathematics, University of New Hampshire, Durham, NH 03824, USA --and--Academy of Mathematics and Systems Science, Chinese Academy of Sciences, Beijing 100190,China}
\email{guweichen14@mails.ucas.ac.cn}

\author[]{Fei Wei$\,^{\dag}$}
\thanks{$^\dag$ Corresponding author (weif@mail.tsinghua.edu.cn)}

\address{Yau Mathematical Sciences Center, Tsinghua University, Beijing 100084, China}

\email{weif@mail.tsinghua.edu.cn}

\maketitle
\begin{abstract}
A vast class of exponential functions are shown to be deterministic.
This class includes functions whose exponents are polynomial-like or ``piece-wise'' close to polynomials after differentiation.
Many of these functions are proved to be disjoint from the M\"obius function. 
\textsc{keywords}. M\"obius function, disjointness, $k$-th difference.
\end{abstract}
\section{Introduction}\label{introduction}
Let $\mu(n)$ be the M\"{o}bius function,
that is,
$\mu(n)$ is $0$ when $n$ is not square free (i.e., divisible by a nontrivial square),
and is $(-1)^{r}$ when $n$ is a product of $r$ distinct primes.
Many problems in number theory can be reformulated in terms of properties of the M\"{o}bius function.
For example,
the Prime Number Theorem is known to be equivalent to $\sum_{n\leq N}\mu(n)=o(N)$.
The Riemann hypothesis holds if and only if $\sum_{n\leq N} \mu(n)= o(N^{\frac{1}{2}+ \epsilon})$ for every $\epsilon>0$.

For a truly random sequence $a_n$ of $-1$ and $1$,
the normalized average $(\sum_{n=1}^N a_n)/N^{\frac12}$ obeys Gaussian law in distribution as $N$ tends to infinity,
which implies that $\sum_{n\le N}a_n=o(N^{\frac12+\epsilon})$.
Under the Riemann hypothesis, the M\"obius function shares this property as an indication that certain randomness may exist in the values of the M\"obius function. It is widely believed that this randomness predicts significant cancellations in the summation of $\mu(n)\xi(n)$ for any ``reasonable" sequence $\xi(n)$. This rather vague principle is known as an instance of the ``M\"obius randomness principle" (see e.g., \cite[Section 13.1]{Iwaniec}). In \cite{Sar}, Sarnak made this principle precise by identifying the notion of ``reasonable"  and proposed the following conjecture.
\begin{conjecture}[Sarnak's  M\"{o}bius Disjointness Conjecture (SMDC)]\label{SMDC}
Let $\xi(n)$ be a deterministic sequence. Then
\[\lim_{N\rightarrow \infty}\frac{1}{N}\sum_{n=1}^{N}\mu(n)\xi(n)=0.\]
\end{conjecture}
Here, we recall the definition of deterministic sequences. Functions from $\mathbb{N}$ to $\mathbb{C}$ are called \emph{arithmetic functions} or\emph{ sequences}. An arithmetic function $f(n)$ is said to be \emph{disjoint} from another one $g(n)$ if $\sum_{n=1}^{N}f(n) \overline{g}(n)=o(N)$. Let $(\mathcal{X},T)$ be a topological dynamical system, that is $\mathcal{X}$ is a compact Hausdorff space and $T:\mathcal{X}\rightarrow \mathcal{X}$ a continuous map.  We say $n\mapsto F(T^{n}x_{0})$ \emph{an arithmetic function realized in the topological dynamical system $(\mathcal{X},T)$}, and $(\mathcal{X},T)$ a \emph{realization} of this arithmetic function. Functions realized in topological dynamical systems of zero topological entropy are called \emph{deterministic}.

In recent years, a lot of progress has been made on SMDC. See \cite{JB,JPT, CE,ELde,FJ,FH,GT,HW,HLW,HWY,HWZ,Kan,DK,LOZ,LW,JP,CM16,Pec,Vee,Wei2, W,xu}, to list a few. In the following, we shall discuss only the results that are more related to this paper. The goal of this article is to show that a class of exponential functions is deterministic and to verify the conjecture of Sarnak for them.
\subsection {Notation}
We use $e(f(n))$ to denote the exponential function $\exp(2\pi\i f(n))$ when $f$ is a real-valued arithmetic function,
where $\i$ is the imaginary unit.
When we say the exponential function of $f(n)$,
we mean $e(f(n))$.
Sometimes we call $e(f(n))$ an $f(n)$ \emph{phase}.

We use $1_{S}$ to denote the indicator of a predicate $S$,
that is $1_{S}=1$ when $S$ is true and $1_{S}=0$ when $S$ is false.
We also denote $1_{A}(n)=1_{n\in A}$ for any subset $A$ of $\mathbb{N}$.
For any finite set $C$,
$|C|$ denotes the cardinality of $C$.

We use $\{x\}$ and $\lfloor x \rfloor$ to denote the fractional part and the integer part of a real number $x$,
respectively.
We use $\|x\|_{\mathbb{R}/\mathbb{Z}}$ to denote the distance between $x$ and the set $\mathbb{Z}$,
i.e., $\|x\|_{\mathbb{R}/\mathbb{Z}}= \min(\{x\},1-\{x\})$.
For ease of notation,
we drop the subscript and write simply $\|x\|$.

The difference operator $\triangle$ is defined on the set of all arithmetic functions,
mapping $f(\cdot)$ to $f(\cdot+1)-f(\cdot)$.
The $k$-th difference operator $\triangle^{k}$ is defined by the composition of $\triangle$ with $k$ times.

For two arithmetic functions $f(n)$ and $g(n)$, $f=o(g)$ means $\lim_{n\rightarrow \infty}f(n)/g(n)=0$; assume that $g(n)\geq 0$ for any $n\in \mathbb{N}$,
$f\ll g$ means that there is an absolute constant $c$ such that $|f|\leq cg$;
$f=g+O(h)$ means $f-g\ll h$.
\subsection{A class of deterministic sequences}
It was investigated in \cite{Ge} and \cite{Wei1}, the set of deterministic sequences is closed under many operations, for example addition, multiplication, inversion, conjugation and translation. This set is also closed under the uniform limit.
Moreover, any continuous function of a deterministic sequence is also deterministic.

\textbf{Motivation.} From the above, it is easy to see that the set of deterministic sequences is closed under the difference operator $\triangle$. As we know, $\triangle$ has an inverse operator $\sigma$ (up to an initial value $f(0)$), mapping $f$ to $\sigma(f): n\mapsto \sum_{j<n}f(j)$.
Then it is natural to ask whether $\sigma(f)$ is also deterministic when $f$ is deterministic. We may assume $f$ is real in
this question because a complex arithmetic function is deterministic if and only if both its real part
and imaginary part are deterministic.
However, $\sigma (f)$ could be unbounded (so the corresponding dynamical system is no longer compact)
even if $f$ is bounded. It is therefore better to consider what kind of properties $f$ has can ensure that
$e(\sigma(f))$ is deterministic. For this question, we give the following answer, which states that $e(\sigma(f))$ is deterministic
when the $k$-th difference $\triangle^k f(n)$  is ``piece-wise'' close to polynomials.


\begin{theorem}\label{a large class of functions with zero anqie entropy}
Let $w$ be a positive integer.
Suppose that $p_1(y),\ldots, p_w(y)$ are polynomials in $\mathbb{R}[y]$,
and $\N=S_1\cup S_2\cup\cdots \cup S_{w}$ is a partition of $\mathbb{N}$ with each $1_{S_{v}}(n)$ deterministic.
Let
\begin{equation}\label{thedefinionofg}
g(n)=\sum_{v=1}^w 1_{S_v}(n) p_v(n).
\end{equation}
Then for any real-valued arithmetic function $f(n)$ satisfying
\begin{equation}\label{conditions need to satisfy12/31}
\lim_{n\rightarrow \infty} \|\triangle^k f(n)-g(n)\|=0
\end{equation}
for some $k\in \N$, $e(f(n))$ and  $e(\sigma(f)(n))$ are deterministic.
\end{theorem}


Now, we explain briefly the main idea to prove the above result. The major tool we use is anqie entropy (of arithmetic functions), which was introduced by Ge in \cite{Ge}. We refer readers to Section \ref{some properties of anqie entropy} for knowledge on anqie entropy.
To prove Theorem \ref{a large class of functions with zero anqie entropy}, we first construct a sequence of arithmetic functions $\{f_{N}(n)\}_{N=0}^{\infty}$ with finite ranges that uniformly converges to $f(n)$. By the lower semi-continuity of anqie entropy (see Proposition \ref{semicontinuity}), it suffices to show that the anqie entropy of $f_{N}$ is zero for $N$ large enough. Note that the anqie entropy of $f_{N}$ (with finite range) is given through the cardinality of different $J$-blocks occurring in it (see formula (\ref{formula in introduction})). So we focus on estimating this cardinality. Our method is to built a one-to-one map from the set of $J$-blocks occurring in the sequence $f_{N}(n)$ to the set of pieces of $\mathbb{R}^{k}$ cut by hyperplanes (for some $k$ depending on $J$). So the cardinality of the first set is bounded by the cardinality of the latter one. And we prove that the second cardinality has polynomial growth (see Lemma \ref{0406lem:fengekongjian}).
We refer readers to Section \ref{Anqie entropy of a certain class of arithmetic functions} for more details.

We also consider whether the
exponential function of any concatenation of polynomials is deterministic.
Before stating the next result,
we first introduce the definition of concatenation of arithmetic functions.
\begin{definition}
Let $0=N_0<N_1<\cdots$ be a sequence of natural numbers with $\lim_{i\rightarrow \infty}(N_{i+1}-N_{i})=\infty$, and $\{f_{i}(n)\}_{i=0}^{\infty}$ be a sequence of arithmetic functions.
We say that $f(n)$ is \textbf{the concatenation of $\{f_{i}(n)\}_{i=0}^{\infty}$}\textbf{ with respect to the sequence} $\{N_{i}\}_{i=0}^{\infty}$ if $f(n)=f_{i}(n)$ when $N_{i}\leq n<N_{i+1}$ for $i=0,1,\ldots$.
\end{definition}

For the exponential functions of concatenations,
we obtain the following result.
\begin{theorem}\label{0601thm2}
Let $g(n)$ be defined as in Theorem \ref{a large class of functions with zero anqie entropy}. Suppose that $\{f_{i}(n)\}_{i=0}^{\infty}$ is a sequence of real-valued arithmetic functions such that
\begin{align}\label{0604shi1}
\lim_{n\rightarrow\infty} \sup_{i\in \N} \; \|\triangle^k f_i(n)-g(n)\|=0
\end{align}
holds for some $k\in \mathbb{N}$.
Then for any concatenation $f(n)$ of $\{f_{i}(n)\}_{i=0}^{\infty}$, $e(f(n))$ is deterministic.
\end{theorem}
As an application of Theorem \ref{0601thm2}, we prove that Sarnak's conjecture implies the averages of the products of M\"{o}bius and certain exponential functions are small in almost all short intervals (Theorem \ref{SMDC implies the distribution of moebius in short interval}).
\subsection{Disjointness of M\"{o}bius from $e(f(n))$ with the $k$-th difference of $f(n)$ tending to zero} The disjointness of M\"{o}bius from exponential functions is important and has been extensively studied in number theory. For example, the disjointness of $\mu(n)$ from $e(n\alpha)$, for any $\alpha\in \mathbb{R}$, is closely related to the estimate of exponential sums in prime variables, from which one can deduce Vinogradov's three primes theorem. It is known that $\mu(n)$ is disjoint from $e(p(n))$ for $p(n)$ a polynomial (or a sub-polynomial) (see \cite{Dav}, \cite[Chapter 6, Theorem 10]{Hua} and \cite[Chapter 3, Theorem 4.6]{Vino}). Moreover, the upper bound of $\sum_{n=1}^N \mu(n)e(p(n))$ has received much attention (see e.g., \cite{BH}, \cite{JL} and \cite{ZL}) partially due to its close connection with the distribution of the zeros of Dirichlet L-functions.

As we have seen from  Theorem \ref{a large class of functions with zero anqie entropy} that $e(f(n))$ is deterministic when $f(n)$ satisfies that there is a $k\in \mathbb{N}$ such that the $k$-th derivative (or the $k$-th difference for the discrete case, see condition (\ref{a condition for the kth difference}) below) tends to zero. Note that SMDC implies that any deterministic sequence is disjoint from the M\"{o}bius function. Motivated by this, we are interested in the following problem.
\begin{problem}\label{the first problem}
Let $f(n)$ be a real-valued arithmetic function such that
\begin{equation}\label{a condition for the kth difference}
\lim_{n\rightarrow \infty}\|\triangle^{k}f(n)\|=0
\end{equation}
for some natural number $k$. Is $\mu(n)$ disjoint from $e(f(n))$?
\end{problem}

Restrictions on the $k$-th derivative or $k$-th difference of $f(n)$ often appear in nontrivial estimate of the exponential sums $\sum_{n=1}^{N} e(f(n))$. For example, Van der Corput's method and the method of exponential pairs (see e.g., \cite{GK}). The latter one usually requires $\triangle f(n)$ to be approximately $cn^{-s}$ for some $c,s>0$. This condition is included in (\ref{a condition for the kth difference}). In this paper, we focus on the sum $\sum_{n=1}^{N}\mu(n)e(f(n))$ under the restriction (\ref{a condition for the kth difference}).

In the following we investigate Problem \ref{the first problem} without assuming SMDC. For the case $k=0$,
it is obvious that $\mu(n)$ is disjoint from such $e(f(n))$ by the Prime Number Theorem.
For the case $k=1$, we have the following result.
\begin{proposition}\label{k=0}
Suppose that $f(n)$ is a real-valued arithmetic function satisfying
\[\lim_{n\rightarrow \infty}\|\triangle f(n)-c\|=0\]
for some constant $c\in \mathbb{R}$.
Then
\[\lim_{N\rightarrow \infty}\frac{1}{N}\sum_{1\leq n\leq N}\mu(n)e(f(n))=0.\]
\end{proposition}
In fact, we have the following general result.
\begin{proposition}\label{k=1 with finite accumulation points}
Suppose that $f(n)$ is a real-valued arithmetic function satisfying that the set $\{e(\triangle f(n)):n=0,1,\ldots\}$ has finitely many limit points, and
\[\lim_{n\rightarrow \infty}\|\triangle^2 f(n)-c\|=0\]
for some constant $c\in \mathbb{R}$.
Then
\[\lim_{N\rightarrow \infty}\frac{1}{N}\sum_{1\leq n\leq N}\mu(n)e(f(n))=0.\]
\end{proposition}
For the case $k\geq 2$, we obtain the following result.
\begin{theorem}\label{k=1rapiddecreasingspeed}
Let $\tau\in (5/8,1)$ and $k\geq 2$. Let $f(n)$ be a real-valued arithmetic function satisfying when $n$ large enough,
\begin{equation}\label{basic relation}
\|\triangle^kf(n)\|\leq \frac{C}{\exp((\log n)^{\tau})}
\end{equation}
for some positive constant $C$.
Then
\[\lim_{N\rightarrow \infty}\frac{1}{N}\sum_{1\leq n\leq N}\mu(n)e(f(n))=0.\]
\end{theorem}
The first ingredient of our proof of the above result is that if the $k$-th difference of $f(n)$ decays to zero, then $e(f(n))$ can be approximated uniformly by certain concatenations of polynomial phases (see Lemma \ref{approximated by polynomials}). The second one is Matom\"{a}ki-Radziwi{\l}{\l}-Tao-Ter\"{a}v\"{a}inen-Ziegler's estimate \cite{MRT20} on averages of the correlation of multiplicative functions with polynomial phases in short intervals.

In the following, we provide a sufficient condition, which is weaker than SMDC (see Corollary \ref{in terms of distribution of polynomial phases}), for Problem \ref{the first problem}.
\begin{proposition}\label{sufficient}
Let $k$ be a given positive integer. Denote $\mathcal{D}_{k}$ by the set of all polynomials in $\mathbb{R}[y]$ of degrees less than $k$. Assume that the following estimate holds,
\begin{equation}\label{conjecture about uniform fourier coefficient}
\lim_{h\rightarrow \infty}\limsup_{X\rightarrow \infty}\frac{1}{Xh}\int_{X}^{2X} \sup_{p(y)\in \mathcal{D}_{k}}\Bigg|\sum_{x\leq n<x+h}\mu(n)e(p(n))\Bigg|dx=0.
\end{equation}
Then for any $f(n)$ satisfying
$\lim_{n\rightarrow \infty}\|\triangle^{k}f(n)\|=0$,
we have
\[\lim_{N\rightarrow\infty}\frac{1}{N}\sum_{1\leq n\leq N}\mu(n)e(f(n))=0.\]
\end{proposition}
\begin{remark}
{\rm All the results listed in this section also hold if $\mu$ is replaced by a more general ``non-pretentious'' 1-bounded (i.e., the $l^{\infty}$-norm is bounded by $1$) multiplicative functions, such as the Liouville function and $\mu(n)\chi(n)$, where $\chi$ is a given Dirichlet character.}
\end{remark}

\section{Preliminaries on anqie entropy of arithmetic functions}
To use tools from operator algebra to study Sarnak's conjecture, Ge introduced the anqie entropy for arithmetic functions \cite{Ge}. For a bounded arithmetic function $f$, its anqie entropy equals the infimum of the topological entropy of all possible realizations of $f$. For functions with finite ranges, the anqie entropy of such a function $f$ is determined by the number of different $J$-blocks appearing in the sequence $\{f(n)\}_{n=0}^{\infty}$.
Specifically, let $\mathcal{B}_{J}(f)$ denote the set of all $J$-blocks occurring in $f$, i.e.,
$\mathcal{B}_{J}(f)=\{(f(n),f(n+1),\ldots,f(n+J-1)):n\geq 0\}$,
then the anqie entropy of $f(n)$ equals
\begin{equation}\label{formula in introduction}
\lim_{J\rightarrow \infty}\frac{\log|\mathcal{B}_{J}(f)|}{J},
\end{equation}
where $|\mathcal{B}_{J}(f)|$ is the cardinality of the set $\mathcal{B}_{J}(f)$ (\cite[Lemma 6.1]{Wei1}).
\subsection{Some properties of anqie entropy}\label{some properties of anqie entropy}
The anqie entropy has many nice properties. Here we list some properties which will be used in this paper.
The following one is about the algebraic operations. Here and in the
sequel, we use $\AE(f)$ to denote the anqie entropy of any bounded arithmetic function $f(n)$.
\begin{proposition}\label{some properties about anqie entropy}
For any bounded arithmetic functions $f,g$ and continuous function $\phi(z)$ in $\mathbb{C}$, we have
\[
\left|{\AE}(f)-{\AE}(g)\right|\leq {\AE}(f\pm g)\leq {\AE}(f)+{\AE}(g),
\]
\[
{\AE}(f\cdot g)\leq {\AE}(f)+{\AE}(g),\;\;\;\AE(\phi(f))\leq \AE(f).
\]
\end{proposition}
The next one is about the lower semi-continuity of anqie entropy.
\begin{proposition}\label{semicontinuity}
If $\{f_{N}(n)\}_{N=0}^{\infty}$ is a sequence of bounded arithmetic functions converging to $f(n)$ uniformly with respect to $n\in \mathbb{N}$, then $\liminf_{N\rightarrow \infty}\AE(f_{N})\geq \AE(f)$.
\end{proposition}
We refer readers to \cite[Section 4]{Ge} and \cite{Wei1} for details of the above two propositions. Arithmetic functions of zero anqie entropy can be realized in topological dynamical systems of zero topological entropy \cite[Section 3]{Ge}. Based on this fact, we have the following result.
\begin{proposition}\label{an equivlent condition of Sarnak's conjecture}
A sequence $\{f(n)\}_{n=0}^{\infty}$ is deterministic if and only if $\AE(f)=0$.
\end{proposition}
\subsection{Anqie entropy of arithmetic functions with finite ranges}\label{Entropy of arithmetic functions with finite ranges}
In this subsection, we show some results on computing anqie entropy of functions with finite ranges, which will be used in the proofs of Theorems \ref{a large class of functions with zero anqie entropy} and \ref{0601thm2}. Let us first recall some basic concepts in symbolic dynamical systems. For a finite set $\mathbb{A}$,
a \emph{block} over $\mathbb{A}$ is a finite sequence of symbols from $\mathbb{A}$.  A \emph{$J$-block}
is a block of length $J$ ($J\geq 1$). For any given (finite or infinite) sequence $x=(x_{0},x_{1},\ldots)$ of symbols
from $\mathbb{A}$, we say that a block $w$ \emph{occurs in} $x$ or $x$ \emph{contains} $w$ if there are
natural numbers $i$, $j$ with $i\leq j$ such that $(x_{i},\ldots,x_{j})=w$.  A \emph{concatenation} of
two blocks $w_{1}=(a_{1},\ldots,a_{k})$ and $w_{2}=(b_{1},\ldots,b_{l})$ over $\mathbb{A}$ is the block
$w_{1}w_{2}=(a_{1},\ldots,a_{k},b_{1},\ldots,b_{l})$.

Now suppose that $f:\mathbb{N}\rightarrow \mathbb{C}$ has finite range. Let $\mathcal{B}_{J}(f)$ denote the set of all $J$-blocks occurring in $f$, i.e.,
$$
\mathcal{B}_{J}(f)=\{(f(n),f(n+1),\ldots,f(n+J-1)):n\geq 0\}.
$$
A $J$-block of the form
$$
(f(lJ),f(lJ+1),\ldots, f(lJ+J-1))
$$
for some $l\in \mathbb{N}$ is called a \emph{regular $J$-block in $f$}. Denote the set of all regular $J$-blocks in $f$ by $\mathcal{B}_{J}^{r}(f)$. A $J$-block, which occurs infinitely many times in the sequence
$\{f(n)\}_{n=0}^{\infty}$, is called an \emph{effective $J$-block in $f$}. Denote the set of all such blocks in $f$ by $\mathcal{B}_{J}^{e}(f)$.  A $J$-block $(a_{0},a_{1},\ldots,a_{J-1})$ is called \emph{regularly effective in $f$} if there are infinitely many natural numbers $l$ such that
\[(a_{0},a_{1},\ldots,a_{J-1})=(f(lJ),...,f(lJ+J-1)).\]
The set of all regularly effective blocks in $f$ is denoted by $\mathcal{B}_{J}^{e,r}(f)$.

For a function $f$ taking finitely many values, as we mentioned previously,
\begin{equation}\label{the definition of entropy}
\AE(f)=\lim_{J\rightarrow \infty}\frac{\log|\mathcal{B}_{J}(f)|}{J}.
\end{equation}
In the following, we show that $\AE(f)$ also can be computed through the cardinality of $\mathcal{B}_{J}^{r}(f)$, $\mathcal{B}_{J}^{e}(f)$ or $\mathcal{B}_{J}^{e,r}(f)$.
\begin{proposition}\label{0330thmyouxiaozhengshang}
Let $f(n)$ be an arithmetic function with finite range. Then \footnote{Although the proof is not hard, we did not
find it in the literature.}
\begin{align}\label{0330shi1}
\AE(f)=
\lim\limits_{J\rightarrow\infty} \frac{\log|\BBB_J^r(f)|} {J}=
\lim\limits_{J\rightarrow\infty} \frac{\log|\BBB^{e,r}_J(f)|} {J}=
\lim\limits_{J\rightarrow\infty} \frac{\log|\BBB_J^e(f)|} {J}.
\end{align}
\end{proposition}
\begin{proof}
We first show that the first equality in equation (\ref{0330shi1}) holds. On one hand,
$\BBB_J^r(f)\subseteq \BBB_J(f)$, so by formula (\ref{the definition of entropy}),
\[
\AE(f)\ge
\limsup\limits_{J\rightarrow\infty} \frac{\log|\BBB_J^r(f)|} {J}\;.
\]
On the other hand, given $J\geq 1$, for any $l\geq 1$ and any $(lJ)$-block $w$ occurring in $f$,
there is a concatenation of certain $l+1$ successive regular $J$-blocks in $f$
containing $w$. Thus $|\BBB_{lJ}(f)|\le J|\BBB^r_J(f)|^{l+1}$. This implies that
$$\AE(f)=\lim_{l\rightarrow \infty}\frac{\log |\mathcal{B}_{lJ}(f)|}{lJ}\leq \frac{\log |\mathcal{B}_{J}^{r}(f)|}{J}.$$
We then have \[\AE(f)\leq \liminf\limits_{J\rightarrow \infty}\frac{\log |\mathcal{B}_{J}^{r}(f)|}{J}.\]
So $\lim_{J\rightarrow \infty}\frac{\log |\mathcal{B}_{J}^{r}(f)|}{J}$ exists and equals $\AE(f)$.

Next we show that the second equality in equation (\ref{0330shi1}) holds, i.e.,
\begin{align}\label{0330shi4}
\lim\limits_{J\rightarrow\infty} \frac{\log|\BBB_J^r(f)|} {J}=
\lim\limits_{J\rightarrow\infty} \frac{\log|\BBB_J^{e,r}(f)|} {J}\;.
\end{align}
Since $\BBB^{e,r}_J(f)\subseteq \BBB_J^r(f)$,
\begin{align}\label{0330shi5}
\lim\limits_{J\rightarrow\infty} \frac{\log|\BBB_J^r(f)|} {J}\ge
\limsup\limits_{J\rightarrow\infty} \frac{\log|\BBB_J^{e,r}(f)|} {J}\;.
\end{align}
So we only need to show that
\begin{align}\label{0330shi6}
\lim\limits_{J\rightarrow\infty} \frac{\log|\BBB_J^r(f)|} {J}\le
\liminf\limits_{J\rightarrow\infty} \frac{\log|\BBB_J^{e,r}(f)|} {J}\;.
\end{align}
In fact,
for any given $J\geq 1$,
since the set $\BBB_J^r(f)\setminus \BBB_J^{e,r}(f)$ is finite, there is an integer $l_{J}\geq 1$ such that all regular $J$-blocks in the set
\[\{(f(nJ),...,f(nJ+J-1)):n\ge l_J\}\]
are regularly effective $J$-blocks in $f$.
Then for each $l>l_{J}$,
there is at most one regular $(lJ)$-block in $f$ which is not a concatenation of regularly effective $J$-blocks in $f$.
This implies $|\BBB^r_{lJ}(f)|\leq |\BBB^{e,r}_{J}(f)|^{l}+1$.
Therefore
$$\AE(f)=\lim_{l\rightarrow \infty} \frac{\log |\BBB^r_{lJ}(f)|}{lJ}\leq \lim_{l\rightarrow \infty} \frac{\log (|\BBB^{e,r}_J(f)|^{l}+1)}{lJ}= \frac{\log |\BBB^{e,r}_J(f)|}{J}$$ holds. Letting $J\rightarrow \infty$, we obtain formula (\ref{0330shi6}).

At last, note that
\[\BBB_J^{e,r}(f)\subseteq \BBB^e_J(f)\subseteq \BBB_J(f),\]
then \[
\lim\limits_{J\rightarrow\infty} \frac{\log|\BBB_J^{e,r}(f)|} {J}\le
\liminf\limits_{J\rightarrow\infty} \frac{\log|\BBB^e_J(f)|} {J}\;,\qquad
\limsup\limits_{J\rightarrow\infty} \frac{\log|\BBB^e_J(f)|} {J}\le
\lim\limits_{J\rightarrow\infty} \frac{\log|\BBB_J(f)|} {J}\;.\]
From formula (\ref{the definition of entropy}) and the second equality in (\ref{0330shi1}), we have
\[
\AE(f)=\lim\limits_{J\rightarrow\infty} \frac{\log|\BBB_J(f)|} {J}=
\lim\limits_{J\rightarrow\infty} \frac{\log|\BBB^{e,r}_J(f)|} {J}.\]
Then $\lim\limits_{J\rightarrow\infty} \frac{\log|\BBB^e_J(f)|} {J}$ exists and equals $\AE(f)$.
So the third equality in equation (\ref{0330shi1}) holds.
\end{proof}
To estimate the cardinality of the set of $J$-blocks occurring in certain sequences, we introduce the following notion.
\begin{definition}\label{definition of pieces}
Let $k,m\geq 1$ be integers and $c_{1},\ldots,c_{m}\in \mathbb{R}$ be constants. Suppose that for $j=1,\ldots,m$, $F_j$ is a non-zero linear function of $x_1,...,x_k$ and $H_j$ is the hyperplane in $\mathbb{R}^{k}$ given by $F_j(x_1,...,x_k)=c_j$. Denote by
\begin{align*}
&H_j^+=\{(x_1,...,x_k)\in \R^k: F_j(x_1,...,x_k)>c_j \},\\  &H_j^-=\{(x_1,...,x_k)\in \R^k: F_j(x_1,...,x_k)<c_j\}.
\end{align*}
A non-empty subset $P$ of $\R^k$ of the following form
\[P=P_1\cap P_2\cap \dots \cap P_m,\]
where each $P_j\in \{H_j^+, H_j^-,H_j\}$, is called  a\textbf{ piece} of $\R^k$ cut by $H_1,...,H_m$.
\end{definition}
In the following lemma we give an upper bound for the cardinality of pieces of $\mathbb{R}^{k}$ cut by hyperplanes.
\begin{lemma}\label{0406lem:fengekongjian}
Let $m, k$ be integers with $m>k\geq 1$. Suppose that
$H_1,...,H_m$ are hyperplanes in $\R^k$. Let $C(H_{1},\ldots,H_{m},k)$ denote the cardinality of pieces of $\R^k$ cut by $H_1,...,H_m$ and $W(m,k)$ denote the maximal value of $C(H_{1},\ldots,H_{m},k)$ when $H_{1},\ldots,H_{m}$ go through all the possible hyperplanes. Then
\[W(m,k)\leq\sum_{j=0}^k 2^j \binom{m}{j}.\]
In particular, $W(m,k)\leq (k+1)2^km^k$.
\end{lemma}
\begin{proof}
Notice that $W(1,k)=3$ for any $k\geq 1$. If we have
\begin{align}\label{0802shi1}
W(m,k)\le W(m-1,k)+2W(m-1,k-1),
\end{align}
then one can easily draw the conclusion by induction on $m$.

Now we show formula (\ref{0802shi1}) holds for $m\geq 2$. Let $\psi$ be the map from the set of all pieces of $\R^k$ cut by $H_1,...,H_m$, denoted by $\mathcal{P}$, onto the set of all pieces of $\R^k$ cut by $H_2,...,H_m$, denoted by $\widetilde{\mathcal{P}}$, given by
\[\psi: P_1 \cap P_2\cap \dots \cap P_m\mapsto P_2\cap \dots \cap P_m.\]
Then for any piece $D\in \widetilde{\mathcal{P}}$,
$|\psi^{-1}(D)|\le 3$.
Now we show that if $|\psi^{-1}(D)|\ge 2$,
then $H_1\cap D$ is nonempty.
In fact,
the only case we need to consider is when both $H_1^+\cap D$ and $H_1^-\cap D$ are nonempty.
In this case,
suppose $p_1\in H_1^+\cap D$ and $p_2 \in H_1^-\cap D$, then there is a point $p_{0}\in  H_1\cap D$ by the convexity of $D$. Hence, $|\mathcal{P}|-|\widetilde{\mathcal{P}}|=C(H_{1},\ldots,H_{m},k)-C(H_{2},\ldots,H_{m},k)$ does not exceed the cardinality of pieces of the form $H_1\cap P_2\cap \dots \cap P_m$ in $\mathcal{P}$ times 2. In the following, we estimate this cardinality.

For each $j=2,\ldots,m$, $H_1\cap H_j$ is a hyperplane in $H_1$,
or empty, or equal to $H_1$.
Denote $G_1,...,G_n$ to be the ones which are hyperplanes in $H_{1}$.
Then $n\le m-1$. We claim that if $H_1\cap P_2\cap \dots \cap P_m=(H_1\cap P_2)\cap \cdots\cap (H_1\cap P_m)\neq \emptyset$, then $(H_1\cap P_2)\cap \cdots\cap (H_1\cap P_m)$ is a piece of $H_{1}$ cut by $G_{1},\ldots,G_{n}$. In fact, for $j=2,\ldots,m$, there are at most three cases. When $H_1\cap H_j=\emptyset$, then $H_{1}\cap H_{j}^{+}=H_{1}$ or $H_{1}\cap H_{j}^{-}=H_{1}$ and then $P_{j}=H_{j}^{+}$ or $H_{j}^{-}$, respectively. When $H_1\cap H_j=H_1$,
then $P_j=H_j$ and $H_1\cap P_j=H_1$. When $H_1\cap H_j=G_{l_{j}}$ is a hyperplane in $H_1$,
then $\{H_1\cap H_j^{+}, H_1\cap H_j^{-}\}=\{G_{l_{j}}^{+}, G_{l_{j}}^{-}\}$. Hence $(H_1\cap P_2)\cap \cdots\cap (H_1\cap P_m)$ is of the form $T_1\cap \cdots \cap T_n$, where each $T_l\in\{G_l,G_l^+,G_l^{-}\}$.

So the cardinality of pieces of the form $H_1\cap P_2\cap \dots \cap P_m$ in $\mathcal{P}$ does not exceed $C(G_{1},\ldots, G_{n},k-1)$ which is at most $W(n,k-1)\le W(m-1,k-1)$.
Therefore the inequality (\ref{0802shi1}) holds
and the proof is completed.
\end{proof}
\section{Proofs of Theorems \ref{a large class of functions with zero anqie entropy} and \ref{0601thm2}}\label{Anqie entropy of a certain class of arithmetic functions}
Recall that $\|x\|=\inf_{m\in \mathbb{Z}}|x-m|=\min\{\{x\},1-\{x\}\}$. Then $\|\cdot\|$ defines a metric on $\mathbb{R}/\mathbb{Z}$ and the topology induced by it on $\mathbb{R}/\mathbb{Z}$ is equivalent to the Euclid topology on the unit circle. The following lemma will be used in this section.
\begin{lemma}\label{find Y(n)}
Let $f,g$ be real-valued arithmetic functions with
\begin{align}\label{0521shi1}
\lim_{n\rightarrow\infty}\|\triangle^k f(n)-g(n)\|=0
\end{align}
for some $k\in \mathbb{N}$.
Then, for any $\varepsilon>0$ and positive integer $m\geq 1$,
there is some $L\in \N$ such that,
whenever $n>L$,
the following holds for any $j$ with $0\leq j\leq m-1$,
$$\|f(n+j)-Y_{n}(n+j)\|\leq \varepsilon,$$
where $Y_{n}(n+j)$ is defined to be $f(n+j)$ when $0\le j\le k-1$ and to be the value
determined by the following linear equations when $k\le j\le m-1$,
\begin{align}\label{0521shi2}
\triangle^k Y_{n}(n+j)=g(n+j)\;,\;\; j=0,1,...,m-k-1\;.
\end{align}
\end{lemma}
\begin{proof}
When $k=0$, the claim is trivial. In the following, we assume $k\geq 1$. We use induction on $m$ to prove the lemma. For $m\leq k$, since $Y_{n}(n+j)=f(n+j)$ for $j=0,1,\ldots,m-1$, choose $L=0$. Then we obtain the claim in the lemma. Assume inductively that the claim holds for some $m_{0}\geq k$. In the following we shall prove the claim holds for $m_{0}+1$ case. By condition (\ref{0521shi1}) and Proposition \ref{formula for the k-th defference},
\[\lim_{n\rightarrow\infty}\|\sum_{l=0}^{k}(-1)^{k-l}\binom{k}{l}f(n+l)-g(n)\|=0.\]
Then for any $\epsilon>0$, there is an $L_{1}>0$ such that whenever $n>L_{1}$,
\[\|\sum_{l=0}^{k}(-1)^{k-l}\binom{k}{l}f(n+m_{0}-k+l)-g(n+m_{0}-k)\|<\epsilon/2,\]
i.e.,
\begin{equation}\label{080601}
\|f(n+m_{0})-\big(g(n+m_{0}-k)-\sum_{l=0}^{k-1}(-1)^{k-l}\binom{k}{l}f(n+m_{0}-k+l)\big)\|<\epsilon/2.
\end{equation}
By the induction hypothesis, there is an $L_{0}\in \N$, whenever $n>L_{0}$,
\begin{equation}\label{080602}
\|f(n+j)-Y_{n}(n+j)\|<\epsilon/2^{k+1},\;j=0,1,\ldots,m_{0}-1.
\end{equation}
Let $L=\max\{L_{0},L_{1}\}$. Then by equations (\ref{080601}) and (\ref{080602}), whenever $n>L$,
\begin{equation}\label{080603}
\|f(n+m_{0})-\big(g(n+m_{0}-k)-\sum_{l=0}^{k-1}(-1)^{k-l}\binom{k}{l}Y_{n}(n+m_{0}-k+l)\big)\|<\epsilon.
\end{equation}
Define $Y_{n}(n+m_{0})$ to be the value determined in the following equation,
\[\sum_{l=0}^{k}(-1)^{k-l}\binom{k}{l}Y_{n}(n+m_{0}-k+l)=g(n+m_{0}-k).\]
Then by equation (\ref{080603}),
\[\|f(n+m_{0})-Y_{n}(n+m_{0})\|<\epsilon.\]
Combing with equation (\ref{080602}), we obtain the claim for $m=m_{0}+1$, completing the induction.
\end{proof}

\begin{reptheorem}{a large class of functions with zero anqie entropy}
Let $w$ be a positive integer.
Suppose that $p_1(y),\ldots, p_w(y)$ are polynomials in $\mathbb{R}[y]$,
and $\N=S_1\cup S_2\cup\cdots \cup S_{w}$ is a partition of $\mathbb{N}$ with each $1_{S_{v}}(n)$ deterministic.
Let
\begin{equation}\label{thedefinionofg}
g(n)=\sum_{v=1}^w 1_{S_v}(n) p_v(n).
\end{equation}
Then for any real-valued arithmetic function $f(n)$ satisfying
\begin{equation}\label{conditions need to satisfy}
\lim_{n\rightarrow \infty} \|\triangle^k f(n)-g(n)\|=0
\end{equation}
for some $k\in \N$, $e(f(n))$ and  $e(\sigma(f)(n))$ are deterministic.
\end{reptheorem}
We use the following strategy to prove $\AE(f(n))=0$. Firstly, we construct a sequence $\{g_{N}(n)\}_{N=0}^{\infty}$ of arithmetic functions with finite ranges to approach $f(n)$ with respect to $\|\cdot\|_{\mathbb{R/\mathbb{Z}}}$. Then it suffices to prove $\AE(g_{N})=0$ for $N$ large enough by Proposition \ref{semicontinuity}. Secondly, we decompose $\mathcal{B}_{{2^{m}}}^{e,r}(g_{N})$, the set of all $2^{m}$-regularly effective blocks occurring in $g_{N}$, into some subsets $\mathcal{A}_{\nu,m}$ according to the ``zero entropy partition'' given in $g(n)$. Thirdly, for each $\nu$, we construct a family of hyperplanes $H_{\nu,0},\ldots,H_{\nu,2^{m}-1}$ in $\mathbb{R}^{q}$ based on the approximation of $f(n)$ to polynomials after differentiation.  Then there is a correspondence between the set of pieces of $\R^q$ cut by $H_{\nu,0},...,H_{\nu, 2^{m}-1}$ and $\mathcal{A}_{\nu,m}$. Precisely, each element in $\mathcal{A}_{\nu,m}$ uniquely determines a piece of $\mathbb{R}^{q}$ cut by $H_{\nu,0},...,H_{\nu,2^{m}-1}$. So $|\mathcal{A}_{\nu,m}|$, moreover $|\mathcal{B}_{{2^{m}}}^{e,r}(g_{N})|$, is bounded by $W(2^m,q)$, which has the polynomial growth rate with respect to $2^{m}$ by Lemma \ref{0406lem:fengekongjian}.
\begin{proof}[Proof of Theorem \ref{a large class of functions with zero anqie entropy}]
In the following, we first prove $e(f(n))$ is deterministic. Then a similar argument leads to $e(\sigma(f)(n))$ deterministic since  $\triangle^k f(n)=\triangle^{k+1} (\sigma (f))(n)$ by the fact that $\sigma$ is the inverse of $\triangle$.

Given $N\geq 1$, by Lemma \ref{find Y(n)}, for each integer $m\geq 1$, there is a sufficiently large $L_m\in \mathbb{N}$ with $2^m|L_m$,
such that whenever $n\ge L_m$, we have
\begin{equation}\label{080606}
\|f(n+j)-Y_{n}(n+j)\|\leq 1/N,\;j=0,\ldots,2^{m}-1,
\end{equation}
where $Y_{n}(n+j)$ is defined to be $f(n+j)$ when $0\le j\le k-1$ and the value determined by the following linear equations when $k\le j\le 2^{m}-1$,
\begin{align}\label{0901shi2}
\triangle^k Y_{n}(n+j)=g(n+j),\;j=0,...,2^{m}-k-1.
\end{align}
Moreover, we may further assume that the sequence $\{L_m\}_{m=0}^{\infty}$ ($L_0=0$) chosen above satisfies $L_{m+1}>L_m$ for each $m\geq 1$.
Let $d_m=(L_{m+1}-L_m)/2^m$. Then the following is a partition of $\mathbb{N}$,
$$\mathbb{N}=\bigcup_{m=0}^\infty \bigcup_{a=0}^{d_m-1} \{L_m+a2^m, L_m+a2^m+1,..., L_m+a2^m+2^m-1\}\;.$$
We define $Y_{L_0+a}(L_0+a )=f(L_{0}+a)$ for $a=0,1,\ldots,L_{1}-1$ and define the arithmetic function $g_{N}$ as follows: for $a=0,\ldots,d_{m}-1$ and $j=0,\ldots,2^{m}-1$, \begin{equation}\label{definition of gN}
g_N(L_m+a2^m+j)=\frac{t}{N}
\end{equation}
when $\{Y_{L_m+a2^m }(L_m+a2^m+j )\}\in [\frac{t}{N},\frac{t+1}{N})$ for some integer $t$ with $0\leq t\leq N-1$.
By formula (\ref{080606}),
\[\|f(L_m+a2^m+j)-g_N(L_m+a2^m+j)\|<2/N.\]
Then $\sup_{n\in \mathbb{N}}\|f(n)-g_N(n)\|<2/N$. Hence $\lim_{N\rightarrow \infty}\sup_{n\in \mathbb{N}}|e(f(n))-e(g_{N}(n))|=0$. To prove $e(f(n))$ deterministic (that is $\AE\big(e(f(n))\big)=0$ by Proposition \ref{an equivlent condition of Sarnak's conjecture}), it suffices to prove $\AE(g_{N}(n))=0$ for $N$ large enough by Propositions \ref{some properties about anqie entropy} and \ref{semicontinuity}.

In the remaining part of this proof, we shall prove $\AE(g_{N})=0$ for any given $N\geq 1$.
Let $\eta$ be the function defined as $\eta(n)=v$ if $n\in S_v $, $v=1.\ldots,w$.
Then the anqie entropy of $\eta$ is zero. Note that $\eta(n)$ has finite range. By formula (\ref{the definition of entropy}),
\[ \lim_{m\rightarrow\infty} \frac{\log |\mathcal{B}_{2^{m}}(\eta)|}{2^{m}}=0,\]
where $\mathcal{B}_{2^{m}}(\eta)$ is the set of all $2^{m}$-blocks occurring in $\eta$.
For any $2^{m}$-block $\nu$ in $\mathcal{B}_{2^{m}}(\eta)$,
denote by
\[\mathcal{A}_{\nu,m}=\{(g_N(n2^{m}),...,g_N(n2^{m}+2^{m}-1)): n\in \mathbb{N},\; n2^{m}\geq L_{m},\; (\eta(n2^{m}),...,\eta(n2^{m}+2^{m}-1))=\nu \}.\]
Recall that $\mathcal{B}_{{2^{m}}}^{e,r}(g_{N})$ denotes the set of all $2^{m}$-regularly effective blocks occurring in $g_{N}$.
Then
\begin{equation}\label{blocks in g2}
\mathcal{B}_{2^{m}}^{e,r}(g_N)\subseteq \bigcup_{\nu\in \mathcal{B}_{2^{m}}(\eta)} \mathcal{A}_{\nu,m}.
\end{equation}
Let $d=\max (deg (p_1),...,deg (p_w))+1$,
where we define $deg(p_i)$ as $-1$ when $p_i=0$. In the following, we estimate the cardinality of $\mathcal{A}_{\nu,m}$ for each given $m$ with $2^{m}> \max(k+1,d)$ and $\nu=(\nu_{0},...,\nu_{2^{m}-1})$ in $\mathcal{B}_{2^{m}}(\eta)$. Denote by $q=k+wd$. If $q=0$ (i.e., $k=0$ and $d=0$), then by the definition, $g_{N}(n)=0$ when $n\geq L_{1}$. So it is easy to see that $\AE(g_{N})=0$. Then $\AE(e(f))=0$. In the following, we may assume that $q\geq 1$.

We first define linear functions $F_{\nu,0},\ldots,F_{\nu,2^{m}-1}$ from $\mathbb{R}^{q}$ to $\mathbb{R}$. Define $F_{\nu,j}(x_{0},x_{1},\ldots,x_{q-1})=x_{j}$ for $j=0,\ldots, k-1$. Assume inductively that we have defined the linear function $F_{\nu,j_{0}}$ for some $j_{0}$ with $k-1\leq j_{0}\leq 2^{m}-2$. Then we define $F_{\nu,j_{0}+1}(x_{0},x_{1},\ldots,x_{q-1})$ to be the function satisfying
\begin{equation}\label{definition of l-th}
\sum_{l=0}^{k}(-1)^{k-l}\binom{k}{l}F_{\nu,j_{0}+1-k+l}(x_{0},x_{1},\ldots,x_{q-1})\\
=P_{\nu_{j_{0}+1-k}}(j_{0}+1-k),
\end{equation}
where
\begin{equation}\label{defintion of P}
P_{v}(j_{0}+1-k)= \sum_{r=0}^{d-1} x_{k+(v-1)d+r} \prod_{\substack{0\le s\le d-1,\\s\neq r }}\frac{(j_{0}+1-k)- s}{ r-s},\;1\leq v\leq w,
\end{equation}
when $d\geq 2$; $P_{v}(j_{0}+1-k)=0$ when $d=0$; $P_{v}(j_{0}+1-k)=x_{k+(v-1)d}$ when $d=1$. By equation (\ref{definition of l-th}), it follows from the inductive assumption that $F_{\nu,j_{0}+1}(x_{0},x_{1},\ldots,x_{q-1})$ is a linear function of $x_{0},x_{1},\ldots,x_{q-1}$.

Next, given $n$ with $(\eta(n2^m), \eta(n2^m+1), ..., \eta(n2^m+2^m-1))=\nu$ and $n2^{m}\geq L_{m}$.
Suppose that $L_{m_{0}}\le n2^m <L_{m_{0}+1}$ for some $m_{0}\ge m$ and $n2^m=n_02^{m_{0}}+ u$ with $0\le u\leq 2^{m_{0}}-2^{m}$.
Taking $y_{j}=\{Y_{n_02^{m_{0}} }(n2^m+j)\}$ for $0\le j\le k-1$ and $y_{k+(v-1)d+r}=\{p_{v}(n2^{m}+r)\}$ for $1\leq v\leq w$, $0\leq r\leq d-1$. In the following, we show that
\begin{equation}\label{the reason to introduce hyperplanes}
\{F_{\nu,j}(y_{0},y_{1},\ldots,y_{q-1})\}=\{Y_{n_02^{m_{0}}}(n2^m+j)\},\;j=0,\dots,2^{m}-1.
\end{equation}
By the definition,
\begin{equation}\label{0831}
F_{\nu,j}(y_{0},y_{1},\ldots,y_{q-1})=\{Y_{n_{0}2^{m_{0}}}(n2^{m}+j)\},\;j=0,\ldots,k-1.
\end{equation}
Plugging $(y_{0},y_{1},\ldots,y_{q-1})$ into equation (\ref{defintion of P}), we have for $j\geq k$,
\[P_{v}(j-k)= \sum_{r=0}^{d-1} \{p_{v}(n2^{m}+r)\} \prod_{\substack{0\le s\le d-1,\\s\neq r }}\frac{(j-k)- s}{r-s}.\]
By Lemma \ref{equivalent conditions}, $\{P_{v}(j-k)\}=\{p_{v}(n2^{m}+j-k)\}$. Then by equation (\ref{definition of l-th}),
\begin{equation}\label{for j greater than k-1}
\{\sum_{l=0}^{k}(-1)^{k-l}\binom{k}{l}F_{\nu,j-k+l}(y_{0},y_{1},\ldots,y_{q-1})\}=\{p_{\nu_{j-k}}(n2^{m}+j-k)\}\:,\;j=k,\ldots,2^{m}-1.
\end{equation}
Using the condition (\ref{0901shi2}), we have
\[\sum_{l=0}^{k}(-1)^{k-l}\binom{k}{l}Y_{n_{0}2^{m_{0}}}(n2^{m}+j-k+l)=g(n2^{m}+j-k)=p_{\nu_{j-k}}(n2^{m}+j-k),\;j=k,\ldots,2^{m}-1.\]
Comparing with equation (\ref{for j greater than k-1}) and by (\ref{0831}), we conclude that equation $(\ref{the reason to introduce hyperplanes})$ holds.

Note that $|j-k|\leq 2^{m}$. Then by equation (\ref{defintion of P}), $|P_{v}(j-k)|\le (d+1)2^{md}$ when $d=0$, or $d\geq 1$, $x_{k+(v-1)d+r}\in [0,1)$, $r=0,\ldots,d-1$. By equation (\ref{definition of l-th}) and Proposition \ref{estimate of the value at every point},
we have
\begin{equation}\label{bound for hyperplanes}
|F_{\nu,j}(y_{0},y_{1},\ldots,y_{q-1})|\leq (k+1)j^{k}(d+1)2^{md}<(k+1)(d+1)2^{mk+md}\;,\;j=0,1,\ldots,2^{m}-1.
\end{equation}
Based on the linear functions $F_{\nu,0},\ldots,F_{\nu,2^{m}-1}$ constructed above, we define a family of hyperplanes $\mathcal{F}=\{H_{M,t,j}:M,t,j\in \mathbb{Z},\;-(k+1)(d+1)2^{m(k+d)}-1\leq M\leq (k+1)(d+1)2^{m(k+d)},\; 0\leq t\leq N-1, \;0\leq j\leq 2^{m}-1 \}$, where
\[H_{M,t,j}= \{(x_0,x_1,...,x_{q-1}) \in \R^{q} \;:\;F_{\nu,j}(x_0,x_1,...,x_{q-1}) =M+\frac{t}{N}\}.\]
Then $|\mathcal{F}|\leq 2(k+2)(d+1)2^{m(k+d+1)}N$. By Lemma \ref{0406lem:fengekongjian}, there are at most $W(|\mathcal{F}|,q)$ pieces of
$\mathbb{R} ^{q}$ cut by the hyperplanes in $ \mathcal{F}$,
where
\begin{equation}\label{bound for number of pieces}
W(|\mathcal{F}|,q)\leq (q+1)2^q \big(2(k+2)(d+1)2^{m(k+d+1)}N\big)^q.
\end{equation}

Now, we are ready to estimate $|\mathcal{A}_{\nu,m}|$. Let $(g_N(n2^{m}),...,g_N(n2^{m}+2^{m}-1))\in \mathcal{A}_{\nu,m}$ with $n2^{m}\geq L_{m}$. Then $(\eta(n2^{m}),...,\eta(n2^{m}+2^{m}-1))=\nu$. Suppose that $L_{m_{0}}\le n2^m <L_{m_{0}+1}$ for some $m_{0}\ge m$ and $n2^m=n_02^{m_{0}}+ u$ with $0\le u\leq 2^{m_{0}}-2^{m}$. Set  \[y_{j}=\{Y_{n_02^{m_{0}} }(n2^m+j)\}\] for $0\le j\le k-1$ and \[y_{k+(v-1)d+r}=\{p_{v}(n2^{m}+r)\}\]for $1\leq v\leq w$, $0\leq r\leq d-1$.
By formula (\ref{bound for hyperplanes}), there are integers $ M_0,M_1,...,M_{2^{m}-1} \in[-(k+1)(d+1)2^{m(k+d)}-1,(k+1)(d+1)2^{m(k+d)}]$ and  $t_0,t_1,\ldots,t_{2^{m}-1}\in [0,N-1]$ such that
\begin{equation}\label{such a peice}
M_j+\frac{t_j}{N}\le F_{\nu,j}(y_{0},y_{1},\ldots,y_{q-1})< M_j+\frac{t_j+1}{N},\; j=0,\ldots,2^{m}-1.
\end{equation}
Note that any two pieces of $\mathbb{R} ^{q}$ cut by hyperplanes in $\mathcal{F}$ are disjoint. Let $P$ be the unique piece containing the point $(y_{0},y_{1},\ldots,y_{q-1})$. Then it is not hard to check that formula (\ref{such a peice}) holds for each $(x_{0},x_{1},\ldots,x_{q-1})\in P$. By equation (\ref{the reason to introduce hyperplanes}), $\{F_{\nu,j}(y_{0},y_{1},\ldots,y_{q-1})\}=\{Y_{n_{0}2^{m_{0}}}(n2^{m}+j)\}\in [\frac{t_j}{N},\frac{t_j+1}{N})$. Then by formula (\ref{definition of gN}), $g_{N}(n2^{m}+j)=\frac{t_{j}}{N}$, $j=0,\ldots,2^{m}-1$.
Moreover, from the above analysis, we conclude that if $n,n'\in \mathbb{N}$ with $(\eta(n2^{m}),\ldots,\eta(n2^{m}+2^{m}-1))=(\eta(n'2^{m}),\ldots,\eta(n'2^{m}+2^{m}-1))=\nu$, such that the corresponding points $(y_{0},y_{1},\ldots,y_{q-1})$ and $(y_{0}',y_{1}',\ldots,y_{q-1}')$ belong to the same piece of $\mathbb{R}^{k}$, then $(g_N(n2^{m}),\ldots,g_N(n2^{m}+2^{m}-1))=(g_N(n'2^{m}),\ldots,g_N(n'2^{m}+2^{m}-1))$. Hence
\[|\mathcal{A}_{\nu,m}|\leq W(|\mathcal{F}|,q).\]
So by equation (\ref{blocks in g2}) and formula (\ref{bound for number of pieces}),
\begin{equation}\label{the number of blocks in approximate functions}
|\mathcal{B}_{2^{m}}^{e,r}(g_{N})|\leq\sum_{\nu \in \mathcal{B}_{2^{m}}(\eta)}|\mathcal{A}_{\nu,m}|\leq (q+1)2^q \big(2(k+2)(d+1)2^{m(k+d+1)}N\big)^q|\mathcal{B}_{2^{m}}(\eta)|.
\end{equation}
Note that $N,k,d,q$ are parameters independent of $m$. By Proposition \ref{0330thmyouxiaozhengshang},
\[\AE(g_{N})=\lim_{m\rightarrow \infty}\frac{\log|\mathcal{B}_{2^{m}}^{e,r}(g_{N})|}{2^{m}}\leq \lim_{m\rightarrow \infty}\frac{\log|\mathcal{B}_{2^{m}}(\eta)|}{2^{m}}=\AE(\eta)=0.\]
So we complete the proof of this theorem.
\end{proof}
Next, we prove Theorem \ref{0601thm2}, which discusses about the anqie entropy of exponential functions of concatenations.
\begin{proof}[Proof of Theorem \ref{0601thm2}]
Given $N\geq 1$. By condition (\ref{0604shi1}) and Lemma \ref{find Y(n)}, for each integer $m\geq 1$, there is a sufficiently large $L_m\in \mathbb{N}$ with $2^m|L_m$,
such that for any $i\in \mathbb{N}$, whenever $n\ge L_m$, we have
\begin{equation}\label{080706}
\|f_{i}(n+j)-Y_{n,i}(n+j)\|\leq 1/N,\;j=0,1,\ldots,2^{m}-1,
\end{equation}
where $Y_{n,i}(n+j)$ equals $f_{i}(n+j)$ when $0\le j\le k-1$, and
is determined by the following linear equations when $k\le j\le 2^{m}-1$,
\begin{align}\label{0521shi2new}
\triangle^k Y_{n,i}(n+j)=g(n+j),\;j=0,1,...,2^{m}-k-1.
\end{align}
Moreover, we may further assume that the sequence $\{L_m\}_{m=0}^{\infty}$ ($L_0=0$) chosen above satisfies $L_{m+1}>L_m$ for each $m\geq 1$.
Let $d_m=(L_{m+1}-L_m)/2^m$. We define $Y_{L_{0}+a,i}(L_{0}+a)=f_{i}(L_{0}+a)$ for $i\in \mathbb{N}$ and $a=0,1,\ldots,L_{1,i}-1$, and a sequence $\{g_{N,i}\}_{i=0}^{\infty}$ of arithmetic functions with finite ranges as follows: for $a=0,1,\ldots,d_{m}-1$ and $j=0,1,\ldots,2^{m}-1$, define
\[g_{N,i}(L_m+a2^m+j)=\frac{t}{N}\]when $\{Y_{L_m+a2^m,i }(L_m+a2^m+j )\}\in [\frac{t}{N},\frac{t+1}{N})$ for some integer $t$ with $0\leq t\leq N-1$. By formula (\ref{080706}),
\[\|f_{i}(L_m+a2^m+j)-g_{N,i}(L_m+a2^m+j)\|<2/N.\]
Then \begin{equation}\label{approximation825}
\|f_{i}(n)-g_{N,i}(n)\|<2/N,\;\text{for}\; \text{any}\; n,\; i\in \mathbb{N}.
\end{equation}
Denote by \[\mathcal{C}_{m}=\{(g_{N,i}(n2^{m}),...,g_{N,i}(n2^{m}+2^{m}-1)):i\in \mathbb{N}, n\in \mathbb{N},\;n2^{m}\geq L_{m}\}.\]
Recall that \begin{equation}
g(n)=\sum_{v=1}^w 1_{S_v}(n) p_v(n),
\end{equation}
where each $1_{S_{v}}(n)$ is deterministic (i.e., $\AE(1_{S_{v}})=0$ by Proposition \ref{an equivlent condition of Sarnak's conjecture}). Define the arithmetic function $\eta$ by $\eta(n)=v$ if $n\in S_v $, $v=1,\ldots,w$. Then $\AE(\eta)=0$.
Recall that $\mathcal{B}_{{2^{m}}}(\eta)$ denotes the set of all $2^{m}$-blocks occurring in $\eta$. Given $\nu\in \mathcal{B}_{{2^{m}}}(\eta)$,
denote by $\widetilde{\mathcal{A}}_{\nu,m}$ the set
\[\{(g_{N,i}(n2^{m}),...,g_{N,i}(n2^{m}+2^{m}-1)):i\in \mathbb{N},\;n\in \mathbb{N},\;n2^{m}\geq L_{m},\;(\eta(n2^{m}),...,\eta(n2^{m}+2^{m}-1))=\nu \}\;.\]
Then
\begin{equation}
\mathcal{C}_{m}\subseteq \bigcup_{\nu\in \mathcal{B}_{2^{m}}(\eta)} \widetilde{\mathcal{A}}_{\nu,m}.
\end{equation}
Let $d=\max (deg (p_1),...,deg (p_w))+1$ and $q=k+wd$.
It follows, from a similar argument to the proof of formula (\ref{the number of blocks in approximate functions}) in Theorem \ref{a large class of functions with zero anqie entropy}, that
\begin{equation}\label{the number of blocks in approximate functions2}
|\mathcal{C}_{m}|\leq\sum_{\nu \in \mathcal{B}_{2^{m}}(\eta)}|\widetilde{\mathcal{A}}_{\nu,m}|\leq (q+1)2^q (2(k+2)(d+1)2^{m(k+d+1)}N)^q|\mathcal{B}_{2^{m}}(\eta)|
\end{equation}
for any $m$ with $2^{m}>\max(k+1,d)$.
Suppose that $f(n)$ is the concatenation of $\{f_{i}(n)\}_{i=0}^{\infty}$ with respect $\{N_{i}\}_{i=0}^{\infty}$.
Let $g_N(n)$ be the concatenation of $\{g_{N,i}\}_{i=0}^{\infty}$ with respect to $\{N_{i}\}_{i=0}^{\infty}$, i.e.,
\[g_N(n)=g_{N,i}(n)\;\;\;\text{ if }\;\;\;N_i\le n< N_{i+1}\;.\]
By formula (\ref{approximation825}),
\[\|g_N(n)-f(n)\|<2/N,\;\text{for}\;\text{any}\; n\in \N.\]
This implies that $\lim_{N\rightarrow \infty}\sup_{n\in \mathbb{N}}|e(g_{N}(n))-e(f(n))|=0$. So by Propositions \ref{some properties about anqie entropy} and \ref{semicontinuity}, to prove $e(f(n))$ deterministic, it suffices to prove $\AE(g_{N})=0$ for $N$ large enough.

In the following, to show $\AE(g_{N})=0$, we estimate $|\mathcal{B}_{2^{m}}^{e,r}(g_{N})|$ for any given $m$ with $2^{m}>\max(k+1,d)$, where $|\mathcal{B}_{2^{m}}^{e,r}(g_{N})|$ is the cardinality of all regularly effective $2^m$-blocks occurring in $g_N$. Let $i_{m}$ be large enough such that $N_{i+1}-N_{i}>2^{m}$ whenever $i>i_{m}$. Let $(g_{N}(n2^{m}),g_{N}(n2^{m}+1),\ldots,g_{N}(n2^{m}+2^{m}-1))$ be a $2^m$-block in $g_N$ with $n2^{m}>\max\{L_{m},N_{i_{m}}\}$. It is easy to see that there are two cases about this block: one case is that there is an $i_{n}$ such that $(g_{N}(n2^{m}),g_{N}(n2^{m}+1),\ldots,g_{N}(n2^{m}+2^{m}-1))=(g_{N,i_{n}}(n2^{m}),g_{N,i_{n}}(n2^{m}+1),\ldots,g_{N,i_{n}}(n2^{m}+2^{m}-1))$; the other case is that there are two integers $i_{n}\geq 0$ and $j_{n}\geq 1$ such that $g_{N}(n2^{m}+j)=g_{N,i_{n}}(n2^{m}+j)$ when $j=0,1,\ldots,j_{n}-1$ and $g_{N}(n2^{m}+j)=g_{N,i_{n}+1}(n2^{m}+j)$ when $j=j_{n},j_{n}+1,\ldots,2^{m}-1$. So $|\mathcal{B}_{2^{m}}^{e,r}(g_{N})|$ is less than or equal to $2^m\times |\mathcal{C}_{m}|^2$.
Note that $N,k,d,q$ are parameters independent of $m$. Then by formula (\ref{the number of blocks in approximate functions2}),
\[\AE(g_{N})=\lim_{m\rightarrow\infty} \frac{\log|\mathcal{B}_{2^{m}}^{e,r}(g_{N})|}{2^m}\le \lim_{m\rightarrow\infty}
\frac{\log(|\mathcal{B}_{2^{m}}(\eta)|^2)}{2^{m}}=2\AE(\eta)=0.\]
Now, we complete the proof of the theorem.
\end{proof}

As an application of Theorem \ref{0601thm2},
we show that under the assumption of SMDC, for any deterministic sequence $\xi(n)$,
$\mu(n)\xi(n)$ does not correlate with $e(f(n))$ in short intervals on average when $f(n)$ satisfies certain conditions.
Precisely,
\begin{theorem}\label{SMDC implies the distribution of moebius in short interval}
Let $g(n)$ be defined as in Theorem \ref{a large class of functions with zero anqie entropy}.
Suppose that $\mathcal{D}$ is a family of real-valued arithmetic functions such that
\begin{align}
\lim_{n\rightarrow\infty} \sup_{f\in \mathcal{D}} \; \|\triangle^k f(n)-g(n)\|=0
\end{align}
holds for some $k\geq 1$. Then SMDC implies that, for any deterministic sequence $\xi(n)$,
\begin{equation}\label{the correlation in short intervals}
\lim_{h\rightarrow \infty}\limsup_{X\rightarrow \infty}\frac{1}{Xh}\int_{X}^{2X} \sup_{f\in\mathcal{D}}\Bigg|\sum_{x\leq n<x+h}\mu(n)\xi(n)e(f(n))\Bigg|dx=0.
\end{equation}
\end{theorem}
\begin{proof}
Let $ \{N_i\}_{i=0}^\infty$ be a given sequence of natural numbers with $N_0=0$ and $\lim_{i\rightarrow\infty} (N_{i+1}-N_i)=\infty$, and $\{f_{i}(n)\}_{i=0}^{\infty}$ be a sequence in $\mathcal{D}$, choose $\{\theta_{i}\}_{i=0}^{\infty}$ as a sequence of numbers in $[0,1)$ such that \[\Bigg|\sum_{N_{i}\leq n<N_{i+1}}\mu(n)\xi(n)e(f_{i}(n))\Bigg|=\Bigg(\sum_{N_{i}\leq n<N_{i+1}}\mu(n)\xi(n)e(f_{i}(n))\Bigg)e(\theta_{i}).\] Define $\widetilde{f_{i}}(n)=f_{i}(n)+\theta_{i}$. Then $\{\widetilde{f_{i}}(n)\}_{i=0}^{\infty}$ satisfies condition (\ref{the correlation in short intervals}) since $k\geq 1$. Let $\widetilde{f}(n)$ be the concatenation of $\{\widetilde{f_{i}}(n)\}_{i=0}^{\infty}$ with respect to the sequence $\{N_{i}\}_{i=0}^{\infty}$. Let $f(n)=e(\widetilde{f}(n))\xi(n)$. By Proposition \ref{some properties about anqie entropy} and Theorem \ref{0601thm2}, $f(n)$ is deterministic. Hence SMDC implies
\begin{align*}
\lim_{m\rightarrow \infty}\frac{1}{N_{m}}\sum_{i=0}^{m-1}\Bigg|\sum_{{N_{i}}\leq n<N_{i+1}}\mu(n)\xi(n)e(f_{i}(n))\Bigg|=0.
\end{align*}
Note that the above equation holds for any sequence $\{N_i\}_{i=0}^\infty$ of natural numbers and any sequence $\{f_{i}(n)\}_{i=0}^{\infty}$ in $\mathcal{D}$.
By Lemma \ref{disjointness and distribution in short interval}, this implies that
\[\lim_{h\rightarrow \infty}\limsup_{X\rightarrow \infty}\frac{1}{Xh}\int_{X}^{2X} \sup_{f\in\mathcal{D}}\Bigg|\sum_{x\leq n<x+h}\mu(n)\xi(n)e(f(n))\Bigg|dx=0,\]
as we claimed.
\end{proof}
If we take $\xi(n)=1_{n\equiv a(mod~q)}(n)$ and $\mathcal{D}$ the set of all polynomials of degrees less than a given positive integer in Theorem \ref{SMDC implies the distribution of moebius in short interval}, then we have the following corollary.
\begin{corollary}\label{in terms of distribution of polynomial phases}
Let $k\geq 1$ be a given integer. Denote by $\mathcal{D}_{k}$ the set of all polynomials in $\mathbb{R}[y]$ of degrees less than $k$. Let $q\geq 1$ and $a\geq 0$ be given integers. Then SMDC implies
\begin{equation}\label{0415}
\lim_{h\rightarrow \infty}\limsup_{X\rightarrow \infty}\frac{1}{Xh}\int_{X}^{2X} \sup_{p(y)\in \mathcal{D}_{k}}\Bigg|\sum_{\substack{x\leq n<x+h\\n\equiv a(mod~q)}}\mu(n)e(p(n))\Bigg|dx=0.
\end{equation}

\end{corollary}
As observed in \cite{Tao}, equation (\ref{0415}) is implied by the local higher order Fourier uniformity conjecture, which is deduced from the Chowla conjecture (\cite{Ch}, see also \cite{Ng}). It is known that Chowla's conjecture implies SMDC. Corollary \ref{in terms of distribution of polynomial phases} shows that equation (\ref{0415}) can be deduced from SMDC.

Here are some results relevant to equation (\ref{0415}).
For $k=1$, equation (\ref{0415}) has been obtained from the work of Matom\"{a}ki-Radziwi{\l}{\l} (\cite{KM}). For $k\geq 2$, it is open whether (\ref{0415}) holds, while recently Matom\"{a}ki-Radziwi{\l}{\l}-Tao-Ter\"{a}v\"{a}inen-Ziegler in \cite{MRT20} established equation (\ref{0415}) when $h=X^{\theta}$ for any fixed $\theta>0$.
Without taking the average on $X$ in equation (\ref{0415}), let $h=X^{\theta}$, the case that $k=1$ and $\theta>0.55$ was obtained by Matom\"{a}ki-Ter\"{a}v\"{a}inen in \cite{MT}; the case that $k=2$ and $\theta> 5/8$ was previously established by Zhan in \cite{zhan} and extended to $\theta>3/5$ in \cite{MT}; the case that $k\geq 2$ and $\theta>2/3$ was obtained by Matom\"{a}ki-Shao in \cite{MS} (also see \cite{huang}, \cite{LZ} for related results on $\Lambda(n)$ instead of $\mu(n)$).

Using a similar idea to the proof of Theorem \ref{a large class of functions with zero anqie entropy}, we obtain the following proposition that gives many characteristic functions with zero anqie entropy.
\begin{proposition}\label{zero entropy set}
Let $p_{1}(y), p_{2}(y)\in \mathbb{R}[y]$. Suppose that
\[S=\{n\in \mathbb{N}:\{p_{1}(n)\}<\{p_{2}(n)\}\}.\]
Then $1_{S}(n)$, the characteristic function defined on $S$, is a deterministic sequence.
\end{proposition}
\begin{proof}
Assume that the degrees of $p_{1}(n)$ and $p_{2}(n)$ are both less than $k$. Then $\triangle^{k}p_{1}(n)=\triangle^{k}p_{2}(n)=0$. In the following, for any given integer $J\geq k+1$, we estimate $|\mathcal{B}_{J}(1_{S})|$, the cardinality of the set of all $J$-blocks occurring in $1_{S}$.

Firstly, we define linear functions $F_{0},F_{1},\ldots, F_{J-1}:\mathbb{R}^{k}\rightarrow \mathbb{R}$. Define $F_{j}(x_{0},x_{1},\ldots,x_{k-1})$ to be $x_{j}$ when $j=0,1,\ldots,k-1$. Assume inductively that we have defined $F_{j_{0}}(x_{0},x_{1},\ldots,x_{k-1})$ for some $j_{0}\geq k-1$. Then define $F_{j_{0}+1}(x_{0},x_{1},\ldots,x_{k-1})$ to be the linear function satisfying
\begin{equation}\label{0808formula1}
\sum_{l=0}^{k}(-1)^{k-l}\binom{k}{l}F_{j_{0}+1-k+l}(x_{0},x_{1},\ldots,x_{k-1})=0,
\end{equation}
equivalently,
\[F_{j_{0}+1}(x_{0},x_{1},\ldots,x_{k-1})=-\sum_{l=0}^{k-1}(-1)^{k-l}\binom{k}{l}F_{j_{0}+1-k+l}(x_{0},x_{1},\ldots,x_{k-1}).\]
By Proposition \ref{formula for the k-th defference}, it is not hard to see that for $i=1,2$ and any $n\in \mathbb{N}$,
\begin{equation}\label{0808formula2}
\{F_{j}(\{p_{i}(n)\},\{p_{i}(n+1)\},\ldots,\{p_{i}(n+k-1)\})\}=\{p_{i}(n+j)\},\; j=0,1,\ldots,J-1.
\end{equation}
By Proposition \ref{estimate of the value at every point},
\begin{equation}\label{0808formula3}
|F_{j}(\{p_{i}(n)\},\{p_{i}(n+1)\},\ldots,\{p_{i}(n+k-1)\})|< (k+1)J^{k},j=0,1,\ldots,J-1.
\end{equation}

Secondly, we define a family of hyperplanes $\mathcal{F}=\{H_{L,j},H_{M,j},H_{N,j}:\; j,L,M,N\in \mathbb{Z},\;0\leq j\leq J-1,\;-2(k+1)J^{k}-1\leq L\leq 2(k+1)J^{k},\;-(k+1)J^{k}-1\leq M,N\leq (k+1)J^{k}\}$ in $\mathbb{R}^{2k}$, where
\[H_{L,j}=\{(x_{0},x_{1},\ldots,x_{2k-1})\in \mathbb{R}^{2k}:F_{j}(x_{0},x_{1},\ldots,x_{k-1})-F_{j}(x_{k},x_{k+1},\ldots,x_{2k-1})=L\},\]
\[H_{M,j}=\{(x_{0},x_{1},\ldots,x_{2k-1})\in \mathbb{R}^{2k}:F_{j}(x_{0},x_{1},\ldots,x_{k-1})=M\},\]
and
\[H_{N,j}=\{(x_{0},x_{1},\ldots,x_{2k-1})\in \mathbb{R}^{2k}:F_{j}(x_{k},x_{k+1},\ldots,x_{2k-1})=N\}.\]
Then
\[|\mathcal{F}|\leq 8(k+2)J^{k+1}.\]
By Lemma \ref{0406lem:fengekongjian}, there are at most $W(|\mathcal{F}|,2k)$ pieces of
$\mathbb{R} ^{2k}$ cut by the hyperplanes in $ \mathcal{F}$,
where
\[
W(|\mathcal{F}|,2k)\leq 8^{2k}(2k+1)2^{2k}(k+2)^{2k}J^{2k(k+1)}.
\]

Let $\mathcal{P}$ be the set of all pieces of $\mathbb{R}^{2k}$ cut by the hyperplanes in $\mathcal{F}$. Recall that $\mathcal{B}_{J}(1_{S})$ denotes the set of all $J$-blocks occurring in $1_{S}(n)$. Denote by $|\mathcal{B}_{J}(1_{S})|=C_{J}$. Suppose that $\mathcal{B}_{J}(1_{S})=\{B_{1},\ldots,B_{C_{J}}\}$. Let $n_{m}=\min\{n\in \mathbb{N}:(1_{S}(n),1_{S}(n+1),\ldots,1_{S}(n+J-1))=B_{m}\}$, for $m=1,\ldots,C_{J} $. Then $B_{m}=(1_{S}(n_{m}),1_{S}(n_{m}+1),\ldots,1_{S}(n_{m}+J-1))$. Define the map \[\psi:\mathcal{B}_{J}(1_{S})\rightarrow \mathcal{P}\] by $\psi(B_{m})=P_{m}$, where $P_{m}$ is the unique piece in $\mathcal{P}$ containing the point $(\{p_{1}(n_{m})\},\ldots,\{p_{1}(n_{m}+k-1)\},\{p_{2}(n_{m})\},\ldots,\{p_{2}(n_{m}+k-1)\})$, for $m=1,\ldots,C_{J} $. Since any two pieces in $\mathcal{P}$ are disjoint, $\psi$ is well-defined.

In the following, we show that $\psi$ is injective. Given $n\in \mathbb{N}$, let
\[(y_{0},\ldots,y_{k-1},y_{k},\ldots,y_{2k-1})=(\{p_{1}(n)\},\ldots,\{p_{1}(n+k-1)\},\{p_{2}(n)\},\ldots,\{p_{2}(n+k-1)\}).\]
Suppose that $(y_{0},\ldots,y_{k-1},y_{k},\ldots,y_{2k-1})\in P$, a piece of $\mathbb{R} ^{2k}$ cut by the hyperplanes in $ \mathcal{F}$. Then by (\ref{0808formula3}), there are integers $M_{0},\ldots,M_{J-1},N_{0},\ldots, N_{J-1}\in [-(k+1)J^{k}-1,(k+1)J^{k}]$ and $L_{0},\ldots,L_{J-1}\in [-2(k+1)J^{k}-1,2(k+1)J^{k}]$, such that
\begin{equation}\label{0808formula4}
L_{j}\leq F_{j}(y_{0},\ldots,y_{k-1})-F_{j}(y_{k},\ldots,y_{2k-1})<L_{j}+1,
\end{equation}
and
\begin{equation}\label{0808formula5}
M_{j}\leq F_{j}(y_{0},\ldots,y_{k-1})<M_{j}+1\;, N_{j}\leq F_{j}(y_{k},\ldots,y_{2k-1})<N_{j}+1.
\end{equation}
Moreover, the above inequalities also hold for each point in $P$. From formulas (\ref{0808formula4}) and (\ref{0808formula5}), it is not hard to see that for any given $j$ with $0\leq j\leq J-1$,
\[\{F_{j}(x_{0},x_{1},\ldots,x_{k-1})\}<\{F_{j}(x_{k},x_{k+1},\ldots,x_{2k-1})\} \]
holds for all $(x_{0},x_{1},\ldots,x_{2k-1})\in P$ (when $L_{j}=M_{j}-N_{j}-1$) or
\[\{F_{j}(x_{0},x_{1},\ldots,x_{k-1})\}\geq \{F_{j}(x_{k},x_{k+1},\ldots,x_{2k-1})\} \]
holds for all $(x_{0},x_{1},\ldots,x_{2k-1})\in P$ (when $L_{j}=M_{j}-N_{j}$). Then by equation (\ref{0808formula2}), we conclude that if $(\{p_{1}(n)\},\ldots,\{p_{1}(n+k-1)\},\{p_{2}(n)\},\ldots,\{p_{2}(n+k-1)\}),(\{p_{1}(n')\},\ldots,\{p_{1}(n'+k-1)\},\{p_{2}(n')\},\ldots,\{p_{2}(n'+k-1)\})\in P$, then $(1_{S}(n),\ldots,1_{S}(n+J-1))=(1_{S}(n'),\ldots,1_{S}(n'+J-1))$. So $\psi$ is injective. Hence $|\mathcal{B}_{J}(1_{S})|=|\psi(\mathcal{B}_{J}(1_{S}))|\leq W(|\mathcal{F}|,2k)\leq 8^{2k}(2k+1)2^{2k}(k+2)^{2k}J^{2k(k+1)}$. Then by formula (\ref{the definition of entropy}),
\[\AE(1_{S})=\lim_{J\rightarrow \infty}\frac{\log|\mathcal{B}_{J}(1_{S})|}{J}=0.\]
So $1_{S}(n)$ is deterministic by Proposition \ref{an equivlent condition of Sarnak's conjecture}.
\end{proof}

As an application of the above proposition, we give the following example that satisfies the condition in Theorem \ref{a large class of functions with zero anqie entropy} with $g(n)=\triangle^2f(n)\neq 0$.
\begin{example}\label{the entropy of a bracket polynomial}
Let $f(n)=\sqrt{3}n\{\sqrt{2}n\}$. Then \[\triangle^2f(n)=\left\{
  \begin{array}{ll}
    2\sqrt{3}(\sqrt{2}-1), & \hbox{$n\in S_{1}$}, \\
    2\sqrt{3}(\sqrt{2}-2), & \hbox{$n\in S_{2}$}, \\
    2\sqrt{3}(\sqrt{2}-1)+\sqrt{3}n, & \hbox{$n\in S_{3}$}, \\
    2\sqrt{3}(\sqrt{2}-2)-\sqrt{3}n, & \hbox{$n\in S_{4}$},
  \end{array}
\right.\]
where $S_{1}=\{n\in \mathbb{N}:\{\sqrt{2}(n+2)\}>\{\sqrt{2}(n+1)\}>\{\sqrt{2}n\}\}$, $S_{2}=\{n\in \mathbb{N}:\{\sqrt{2}(n+2)\}<\{\sqrt{2}(n+1)\}<\{\sqrt{2}n\}\}$, $S_{3}=\{n\in \mathbb{N}:\{\sqrt{2}(n+2)\}>\{\sqrt{2}(n+1)\},\;\{\sqrt{2}(n+1)\}<\{\sqrt{2}n\}\}$, $S_{4}=\{n\in \mathbb{N}:\{\sqrt{2}(n+2)\}<\{\sqrt{2}(n+1)\},\;\{\sqrt{2}(n+1)\}>\{\sqrt{2}n\}\}$.
\end{example}
\section{The M\"{o}bius disjointness of $e(f(n))$ with the $k$-th difference of $f(n)$ tending to zero}\label{disjointness}
In this section, we shall study the disjointness of the M\"{o}bius function from exponential functions of arithmetic functions with the $k$-th differences tending to a constant (Propositions \ref{k=0}, \ref{k=1 with finite accumulation points}, \ref{sufficient}, and Theorem \ref{k=1rapiddecreasingspeed}). We first show the following property that arithmetic functions with $k$-th differences tending to zero can be approximated by certain concatenations of polynomials of degrees less than $k$.
\begin{lemma}\label{approximated by polynomials}
Suppose that $f(n)$ is a real-valued arithmetic function such that
\[\lim_{n\rightarrow \infty}\|\triangle^{k}f(n)\|=0,\]
for some integer $k\geq 1$. Then for any integer $N\geq 1$, there is an increasing sequence $\{N_{i}\}_{i=0}^{\infty}$ of natural numbers with $N_{0}=0$ and $\lim_{i\rightarrow \infty}(N_{i+1}-N_{i})=\infty$, and a sequence $\{p_{i}(y)\}_{i=0}^{\infty}$ in $\mathbb{R}[y]$ of degrees less than $k$, such that
\begin{equation}\label{0814formula5}
\|f(n)-g_{N}(n)\|\leq 1/N,\;\text{for}\;\text{any}\;n\in \mathbb{N},
\end{equation}
where $g_{N}$ is the concatenation of $\{p_{i}(n)\}_{i=0}^{\infty}$ with respect to $\{N_{i}\}_{i=0}^{\infty}$.
\end{lemma}
\begin{proof}
By Lemma \ref{find Y(n)}, for each integer $m\geq 1$ and $N\geq 1$, there is a sufficiently large $L_m\in \mathbb{N}$ with $2^m|L_m$ such that, whenever $n\ge L_m$, we have
\begin{equation}\label{0814formula1}
\|f(n+j)-Y_{n}(n+j)\|\leq 1/N,\;j=0,1,\ldots,2^{m}-1,
\end{equation}
where $Y_{n}(n+j)$ is defined to be $f(n+j)$ when $0\le j\le k-1$ and the value
determined by the following linear equations when $k\le j\le 2^{m}-1$,
\[\triangle^k Y_{n}(n+j)=0,\;j=0,1,\ldots,2^{m}-k-1.\]
It is not hard to check that
\begin{equation}\label{0814formula3}
Y_{n}(n+j)=\sum_{l=0}^{k-1}f(n+l)\prod_{t=0,t\neq l}^{k-1}\frac{j-t}{l-t},\;j=0,1,\ldots,2^{m}-1.
\end{equation}
We may further assume that the sequence $\{L_m\}_{m=0}^{\infty}$ ($L_0=0$) chosen above satisfies $L_{m+1}>L_m$ for each $m$.
Let $d_m=(L_{m+1}-L_m)/2^m$. Then the following is a partition of $\mathbb{N}$.
$$\mathbb{N}=\bigcup_{m=0}^\infty \bigcup_{q=0}^{d_m-1} \{L_m+q2^m, L_m+q2^m+1,..., L_m+q2^m+2^m-1\}.$$
Choose the sequence $\{N_{i}\}_{i=0}^{\infty}$ with $N_{0}<N_{1}<N_{2}<\cdot\cdot\cdot$ such that
\[\{N_{0},N_{1},\cdot\cdot\cdot\}=\{L_{m}+q2^{m}: m\in \mathbb{N},\;0\leq q\leq d_{m}-1\}.\]
Define
\begin{equation}\label{definition of p}
p_{i}(n)=\sum_{l=0}^{k-1}f(N_{i}+l)\prod_{t=0,t\neq l}^{k-1}\frac{n-N_{i}-t}{l-t}.
\end{equation}
It is a polynomial of degree less than $k$. Let
\[g_{N}(n)=p_{i}(n),\;\text{when}\;N_{i}\leq n\leq N_{i+1}-1.\]
By formula (\ref{0814formula1}) and equation (\ref{0814formula3}), we obtain formula (\ref{0814formula5}).
\end{proof}

\begin{lemma}\label{subsequence is equivalent all natural number}
Let $\{N_{i}\}_{i=0}^{\infty}$ be an increasing sequence of natural numbers with $N_{0}=0$ and $\lim_{i\rightarrow \infty}(N_{i+1}-N_{i})= \infty$. Let $\{p_{i}(y)\}_{i=0}^{\infty}$ be a sequence in $\mathbb{R}[y]$ with degrees less than $k$ for some positive integer $k$. Suppose that $f(n)$ is the concatenation of $\{p_{i}(n)\}_{i=0}^{\infty}$ with respect to $\{N_{i}\}_{i=0}^{\infty}$. Let $q$ be a positive integer and $0\leq a \leq q-1$.
Then
\begin{equation}\label{0817formula11}
\lim_{N\rightarrow \infty}\frac{1}{N}\sum_{\substack{1\leq n\leq N\\n\equiv a(mod~q)}}\mu(n)e(f(n))=0
\end{equation}
if and only if
\begin{equation}\label{0817formula12}
\lim_{m\rightarrow \infty}\frac{1}{N_{m}}\sum_{i=0}^{m-1}\sum_{\substack{N_{i}\leq n\leq N_{i+1}-1\\n\equiv a(mod~q)}}\mu(n)e(p_{i}(n))=0.
\end{equation}
\end{lemma}
\begin{proof}
It is obvious that $(\ref{0817formula11})\Rightarrow (\ref{0817formula12})$. We now show $(\ref{0817formula12}) \Rightarrow(\ref{0817formula11})$. In this process, we need to use a classical result (see \cite{Dav} for $k=2$ and \cite[Chapter 6, Theorem 10]{Hua} for $k>2$) stated as follows,
\[\sup_{\substack{p(y)\in \mathbb{R}[y] \\deg(p(y))<k}}\Bigg|\sum_{\substack {1\leq n\leq N\\n\equiv a(mod~q)}}\mu(n)e(p(n))\Bigg|\ll N/\log N,\]
where the implied constant at most depends on $k$ and $q$.
By the above inequality and equation (\ref{0817formula12}), for any given $\epsilon>0$, there is a positive integer $M$ such that whenever $m\geq M$ and $N\geq N_{M}$, we have
\begin{equation}\label{0817formula115}
\sum_{i=0}^{m-1}\sum_{\substack{N_{i}\leq n\leq N_{i+1}-1\\n\equiv a(mod~q)}}\mu(n)e(p_{i}(n))<(\epsilon/2)N_{m},
\end{equation}
and
\begin{equation}\label{0817formula116}
\sup_{i\in \mathbb{N}}\Bigg|\sum_{\substack {1\leq n\leq N\\n\equiv a(mod~q)}}\mu(n)e(p_{i}(n))\Bigg|<(\epsilon/4)N.
\end{equation}
Let $N\geq N_{M}$. Choose an appropriate $l\geq M$ with $N_{l}\leq N\leq N_{l+1}-1$. Then
\begin{align*}
  \Bigg|\sum_{\substack{n=1\\n\equiv a(mod~q)}}^{N}\mu(n)e(f(n))\Bigg|\leq & \Bigg|\sum_{i=0}^{l-1}\sum_{\substack{N_{i}\leq n< N_{i+1}\\n\equiv a(mod~q)}}\mu(n)e(p_{i}(n))\Bigg|+\Bigg|\sum_{\substack{N_{l}\leq n\leq N\\n\equiv a(mod~q)}}\mu(n)e(p_{l}(n))\Bigg| \\
  <&(\epsilon/2)N_{l}+(\epsilon/2)N<\epsilon N.
\end{align*}
Hence we obtain (\ref{0817formula11}).
\end{proof}
We now prove Proposition \ref{sufficient}, which gives a sufficient condition of disjointness between the M\"{o}bius function and exponential functions of arithmetic functions with the $k$-th differences tending to $0$.
\begin{proof}[Proof of Proposition \ref{sufficient}]
By Lemma \ref{approximated by polynomials}, there is a sequence $\{g_{M}(n)\}_{M=1}^{\infty}$ such that \[\lim_{M\rightarrow \infty}\sup_{n\in \mathbb{N}}|e(g_{M}(n))-e(f(n))|=0,\] where $g_{M}(n)$ is a certain concatenation of polynomials. Under the assumption of equation (\ref{conjecture about uniform fourier coefficient}), by Lemmas \ref{disjointness and distribution in short interval} and \ref{subsequence is equivalent all natural number}, we obtain
\[\lim_{N\rightarrow\infty}\frac{1}{N}\sum_{n=1}^{N}\mu(n)e(g_{M}(n))=0.\]
So
\[\lim_{N\rightarrow\infty}\frac{1}{N}\sum_{n=1}^{N}\mu(n)e(f(n))=0\]
completing the proof.
\end{proof}

Next, we shall prove Propositions \ref{k=0}, \ref{k=1 with finite accumulation points}. It is easy to see that Proposition \ref{k=0} directly follows from Proposition  \ref{k=1 with finite accumulation points}, so in the following, we just give the proof of Proposition \ref{k=1 with finite accumulation points}. Before proving it, we need some preparations. The following Dirichlet's approximation theorem is classical and well-known (see e.g., \cite[Section 8.2]{ECT}), which is proved via the fact that if there are $m+1$ points contained in $m$ regions, then there must be at least two points lie in the same region.
\begin{lemma}\label{Dirichlet theorem}
Given $L$ real numbers $\theta_{1},\ldots, \theta_{L}$ and a positive integer $q$, then we can find an integer $t\in [1,q^{L}]$, and integers $a_{1},\ldots,a_{L}$ such that $|t\theta_{j}-a_{j}|\leq 1/q$, $j=1,2,\ldots,L$.
\end{lemma}
The following ``asymptotical periodicity'' of the concatenation of certain linear phases will be used in the proof of Proposition \ref{k=1 with finite accumulation points}.
\begin{lemma}\label{approximation}
Let $\{N_{i}\}_{i=0}^{\infty}$ be an increasing sequence of integers with $N_{0}=0$ and $\lim_{i\rightarrow\infty}(N_{i+1}-N_{i})=\infty$. Suppose that $\alpha_{0},\alpha_{1},\ldots$ are real numbers such that the sequence $\{e(\alpha_{i})\}_{i=0}^{\infty}$ has finitely many limit points. Assume that $f(n)=e(n\alpha_{i})e(\beta_{i})$ when $N_{i}\leq n< N_{i+1}$ for $i=0,1,2,\ldots$, where $\beta_{0},\beta_{1},\ldots$ are real numbers. Then there is a $\delta$ with $0<\delta<1$ and a sequence $\{n_{j}\}_{j=0}^{\infty}$ of positive integers with $\lim_{j\rightarrow \infty}n_{j}=\infty$ such that
\[\lim_{j\rightarrow \infty}\sum_{l=1}^{n_{j}^{\delta}}\limsup_{N\rightarrow \infty}\frac{1}{N}\sum_{n=0}^{N-1}|f(n+ln_{j})-f(n)|^2=0.\]
\end{lemma}

\begin{proof}
Suppose that the limit points of $\{e(\alpha_{i})\}_{i=0}^{\infty}$ are $e(\theta_{1}),\ldots,e(\theta_{L})$ for some integer $L\geq 1$. Let $q_{j}=(j+6)^2\pi^2$, $j\geq 0$. By Lemma \ref{Dirichlet theorem}, we can find an integer $n_{j}$ with $1\leq n_{j}\leq q_{j}^{L}$ such that $|\theta_{s}-\frac{a_{s,j}}{n_{j}}|\leq\frac{1}{n_{j}q_{j}}$, where $a_{s,j}$ is some integer for $s=1,\ldots,L$. It is not hard to check that the sequence $\{n_{j}\}_{j=0}^{\infty}$ can be chosen to satisfy $\lim_{j\rightarrow \infty}n_{j}=\infty$. Moreover, there is an $i_{0}$ such that when $i\geq i_{0}$ we can choose an $s\in \{1,2,\ldots,L\}$ satisfying $\|\alpha_{i}-\theta_{s}\|<\frac{1}{n_{j}q_{j}}$. Then for $i\geq i_{0}$, $\|n_{j}\alpha_{i}\|<\frac{2}{q_{j}}$ and $|e(n_{j}\alpha_{i})-1|=2|\sin(\pi n_{j}\alpha_{i})|\ll \frac{1}{q_{j}}$. So, for any given $n_{j}$, \begin{eqnarray*}
  &&\sum_{l=1}^{n_{j}^{\frac{1}{2L}}}\limsup_{N\rightarrow \infty}\frac{1}{N}\Bigg(\sum_{i=0}^{m-1}\sum_{n=N_{i}}^{N_{i+1}-1}|f(n+ln_{j})-f(n)|^2+\sum_{n=N_{m}}^{N}|f(n+ln_{j})-f(n)|^2\Bigg)\\
&&=\sum_{l=1}^{n_{j}^{\frac{1}{2L}}}
\limsup_{N\rightarrow \infty}\frac{1}{N}\Bigg(\sum_{i=0}^{m-1}\sum_{n=N_{i}}^{N_{i+1}-ln_{j}}|f(n+ln_{j})-f(n)|^2+\sum_{n=N_{m}}^{N-ln_{j}}|f(n+ln_{j})-f(n)|^2\Bigg) \\
 && =\sum_{l=1}^{n_{j}^{\frac{1}{2L}}}\limsup_{N\rightarrow \infty}\frac{1}{N}\Bigg(\sum_{i=0}^{m-1}\sum_{n=N_{i}}^{N_{i+1}-ln_{j}}|e(ln_{j}\alpha_{i})-1|^2+\sum_{n=N_{m}}^{N-ln_{j}}|e(ln_{j}\alpha_{m})-1|^2\Bigg)\\
 &&=\sum_{l=1}^{n_{j}^{\frac{1}{2L}}}\limsup_{m\rightarrow \infty}\frac{1}{N}\Bigg(\sum_{i=i_{0}}^{m-1}\sum_{n=N_{i}}^{N_{i+1}-ln_{j}}|e(ln_{j}\alpha_{i})-1|^2+\sum_{n=N_{m}}^{N-ln_{j}}|e(ln_{j}\alpha_{m})-1|^2\Bigg)\\
&&\ll \sum_{l=1}^{n_{j}^{\frac{1}{2L}}}\frac{l^2}{q_{j}^2}\ll \frac{1}{q_{j}^{\frac{1}{2}}}\rightarrow 0,\;j\rightarrow \infty.
\end{eqnarray*}
The claimed result follows by choosing $\delta=\frac{1}{2L}$.
\end{proof}
To prove Proposition \ref{k=1 with finite accumulation points}, we also need the following result of the second author \cite[Lemma 4.1]{W} on the self-correlation of $\mu(n)e(P(n))$ in short arithmetic progressions.
\begin{lemma}\label{the distribution of mobius in short arithmetic interval}
Let $s\geq 1$ and $h\geq 3$ be integers. Suppose that $P(x)\in \mathbb{R}[x]$ is of degree $d\geq 0$.
\begin{equation}\label{has relation with k}
\limsup_{N\rightarrow \infty}\frac{1}{N}\sum_{n=1}^{N}\Bigg|\frac{1}{h}\sum_{l=1}^{h}\mu(n+ls)e\big(P(n+ls)\big)\Bigg|^2\ll \frac{s}{\varphi(s)}\frac{\log\log h}{\log h},
\end{equation}
where $\varphi$ is the Euler totient function and the implied constant depends on $d$ at most.
\end{lemma}
\begin{proof}[Proof of Proposition \ref{k=1 with finite accumulation points}]
Let $f_{0}(n)=\frac{c}{2}n^2+c_{1}n+c_{0}$ for $n\in \mathbb{N}$. Then $\triangle^2 f_{0}(n)=c$ and $\lim_{n\rightarrow \infty}\|\triangle^2(f-f_{0})(n)\|=0$. For any given $\epsilon>0$, by Lemma \ref{approximated by polynomials} and equation (\ref{definition of p}), there is an increasing sequence $\{N_{i}\}_{i=0}^{\infty}$ with $N_{0}=0$ and $\lim_{i\rightarrow \infty}(N_{i+1}-N_{i})=\infty$, and a function $g_{\epsilon}(n)$ defined by $g_{\epsilon}(n)=n(f(N_{i}+1)-f(N_{i}))+(N_{i}+1)f(N_{i})-f(N_{i}+1)N_{i}$ when $N_{i}\leq n<N_{i+1}$ for $i=0,1,\ldots$, such that
\begin{equation}\label{0817afternoonformula2}
\sup_{n\in \mathbb{N}}|e(f(n))-e(f_{0}(n)+g_{\epsilon}(n))|<\epsilon.
\end{equation}
We claim that
\begin{equation}\label{formula0415afternoon}
\lim_{m\rightarrow \infty} \frac{1}{N_{m}}\sum_{i=0}^{m-1}\Bigg|\sum_{N_i\leq n <N_{i+1}}
\mu(n)e(f_{0}(n))e(g_{\epsilon}(n))\Bigg|= 0.
\end{equation}
Using Lemma \ref{subsequence is equivalent all natural number},
\[\lim_{N\rightarrow \infty}\frac{1}{N}\sum_{1\leq n\leq N}\mu(n)e(f_{0}(n)+g_{\epsilon}(n))=0.\]
So by formula (\ref{0817afternoonformula2}), we have for $N$ large enough,
\[\frac{1}{N}\Bigg|\sum_{1\leq n\leq N}\mu(n)e(f(n))\Bigg|<2\epsilon.\]
Letting $\varepsilon\rightarrow 0$, we then obtain the statement in this proposition.

We are left to prove claim (\ref{formula0415afternoon}). Write $\alpha_{i}=f(N_{i}+1)-f(N_{i})$. Choose $\{\beta_{i}\}_{i=0}^{\infty}$ as a sequence of real numbers such that
\[ \Bigg|\sum_{N_i\leq n <N_{i+1}}
\mu(n)e(f_{0}(n))e(g_{\epsilon}(n))\Bigg|=\sum_{N_i\leq n <N_{i+1}}
\mu(n)e(f_{0}(n))e(n\alpha_i)e(\beta_{i}).\]
Define $g(n)$ to be $e(n\alpha_{i})e(\beta_{i})$ when $N_{i}\leq n< N_{i+1}$, $i=0,1,\ldots$. So it suffices to prove that
\begin{equation}\label{0817formula1}
\lim_{N\rightarrow \infty}\frac{1}{N}\sum_{n=1}^{N}\mu(n)e(f_{0}(n))g(n)=0.
\end{equation}
Note that the set $\{e\big(\alpha_{i}\big):i=0,1,\ldots\}$ has finitely many limit points. By Lemma \ref{approximation}, there is a $\delta$ with $0<\delta<1$ and a sequence $\{n_{j}\}_{j=0}^{\infty}$ of positive integers with $\lim_{j\rightarrow\infty}n_{j}=\infty$ such that
\begin{equation}\label{0907}
\lim_{j\rightarrow \infty}\frac{1}{n_{j}^{\delta}}\sum_{l=1}^{n_{j}^{\delta}}\limsup_{N\rightarrow \infty}\frac{1}{N}\sum_{n=0}^{N-1}|g(n+ln_{j})-g(n)|^2=0.
\end{equation}
Note that $\frac{s}{\varphi(s)}\ll \log\log s$. By Lemma \ref{the distribution of mobius in short arithmetic interval}, for any $n_{j}$,
\begin{equation}\label{0907formula2}
\limsup_{N\rightarrow \infty}\frac{1}{N}\sum_{n=1}^{N}\Big|\frac{1}{n_{j}^{\delta}}\sum_{l=1}^{n_{j}^{\delta}}\mu(n+ln_{j})e(f_{0}(n+ln_{j}))\Big|^2\ll \big(\log\log n_{j}\big)\frac{\log\log n_{j}^{\delta}}{\log n_{j}^{\delta}}.
\end{equation}
Note that the right side of the above inequality tends to zero as $j\rightarrow \infty$. Hence for any given $\epsilon>0$, by formulas (\ref{0907}) and (\ref{0907formula2}), there are positive integers  $j_{0}$ and $M_{0}$ such that whenever $N>M_{0}$,
\[\frac{1}{n_{j_{0}}^{\delta}}\sum_{l=1}^{n_{j_{0}}^{\delta}}\frac{1}{N}\sum_{n=1}^{N}|g(n+ln_{j_{0}})-g(n)|^2<\epsilon^2/9\]
and
\[
\frac{1}{N}\sum_{n=1}^{N}\Big|\frac{1}{n_{j_{0}}^{\delta}}\sum_{l=1}^{n_{j_{0}}^{\delta}}\mu(n+ln_{j_{0}})e(f_{0}(n+ln_{j_{0}}))\Big|^2<\epsilon^2/9.
\]
Then by the Cauchy-Schwarz inequality,
\begin{equation}\label{0817formula2}
\frac{1}{n_{j_{0}}^{\delta}}\sum_{l=1}^{n_{j_{0}}^{\delta}}\frac{1}{N}\sum_{n=1}^{N}
|g(n+ln_{j_{0}})-g(n)|<\epsilon/3,
\end{equation}
and
\begin{equation}\label{0817formula3}
\frac{1}{N}\sum_{n=1}^{N}\Bigg|\frac{1}{n_{j_{0}}^{\delta}}\sum_{l=1}^{n_{j_{0}}^{\delta}}\mu(n+ln_{j_{0}})e(f_{0}(n+ln_{j_{0}}))\Bigg|<\epsilon/3.
\end{equation}
Observe that
\[\frac{1}{N}\sum_{n=1}^{N}\mu(n)e\big(f_{0}(n)\big)g(n)=\frac{1}{N}\sum_{n=1}^{N}\mu(n+ln_{j_{0}})e\big(f_{0}(n+ln_{j_{0}})\big)g(n+ln_{j_{0}})+O(\frac{ln_{j_{0}}}{N}).\]
Write $h_{j_{0}}=n_{j_{0}}^{\delta}$. Then there is a positive integer $M_{1}$ such that whenever $N>M_{1}$,
\begin{equation}\label{0817formula4}
\Bigg|\frac{1}{N}\sum_{n=1}^{N}\mu(n)e\big(f_{0}(n)\big)g(n)-\frac{1}{N}\sum_{n=1}^{N}\frac{1}{h_{j_{0}}}\sum_{l=1}^{h_{j_{0}}}\mu(n+ln_{j_{0}})
e\big(f_{0}(n+ln_{j_{0}})\big)
g(n+ln_{j_{0}})\Bigg|<\epsilon/3.
\end{equation}
Let $M_{2}=\max\{M_{0},M_{1}\}$. Then by formulas (\ref{0817formula2}), (\ref{0817formula3}) and (\ref{0817formula4}), for $N>M_{2}$, we have
\begin{align*}
\Bigg|\frac{1}{N}\sum_{n=1}^{N}\mu(n)e\big(f_{0}(n)\big)g(n)\Bigg|<&\Bigg|\frac{1}{N}\sum_{n=1}^{N}\frac{1}{h_{j_{0}}}\sum_{l=1}^{h_{j_{0}}}
\mu(n+ln_{j_{0}})
e\big(f_{0}(n+ln_{j_{0}})\big)
g(n+ln_{j_{0}})\Bigg|+\epsilon/3 \nonumber\\
\leq &\Bigg|\frac{1}{N}\sum_{n=1}^{N}\frac{1}{h_{j_{0}}}\sum_{l=1}^{h_{j_{0}}}\mu(n+ln_{j_{0}})
e\big(f_{0}(n+ln_{j_{0}})\big)
\big(g(n+ln_{j_{0}})-g(n)\big)\Bigg|\nonumber\\
&+\Bigg|\frac{1}{N}\sum_{n=1}^{N}\frac{1}{h_{j_{0}}}\sum_{l=1}^{h_{j_{0}}}\mu(n+ln_{j_{0}})
e\big(f_{0}(n+ln_{j_{0}})\big)g(n)
\Bigg|+\epsilon/3 \nonumber\\
\leq &\frac{1}{N}\sum_{n=1}^{N}\frac{1}{h_{j_{0}}}\sum_{l=1}^{h_{j_{0}}}
|g(n+ln_{j_{0}})-g(n)|\nonumber\\&+\frac{1}{N}\sum_{n=1}^{N}\Bigg|\frac{1}{h_{j_{0}}}\sum_{l=1}^{h_{j_{0}}}\mu(n+ln_{j_{0}})
e\big(f_{0}(n+ln_{j_{0}})\big)\Bigg|
+\epsilon/3\nonumber\\<&\epsilon.
\end{align*}
Then we obtain equation (\ref{0817formula1}).
\end{proof}
In the remaining part of this section, we shall prove Theorem \ref{k=1rapiddecreasingspeed}. The major ingredient of our proof is Matom\"{a}ki-Radziwi{\l}{\l}-Tao-Ter\"{a}v\"{a}inen-Ziegler's recent work \cite{MRT20} on averages of the correlation between multiplicative functions and polynomial phases in short intervals. To state this result, we shall use the following distance function of Granville and Soundararajan,
\[\mathbb{D}(g(n),n \mapsto n^{it};X):=\Big(\sum_{p\leq X}\frac{1-\text{Re}(g(p)p^{-it})}{p}\Big)^{\frac{1}{2}}\]
for any multiplicative function $g(n)$ with $|g(n)|\leq 1$ for all $n\in\mathbb{N}$. This distance function was introduced in \cite{GS07} to measure the pretentiousness between $g(n)$ and $n \mapsto n^{it}$. Throughout this section, define
\[M(g;X,Q) := \inf_{\substack{\chi~mod~q \\ 1\leq q\leq Q\\|t|\leq X}} \mathbb{D}(g(n)\chi(n),n \mapsto n^{it};X)^2
=\inf_{\substack{ \chi~mod~q \\ 1\leq q\leq Q \\|t|\leq X}} \sum_{p\leq X} \frac{1-\Re(g(p)\chi(p)p^{-it})}{p}.\]
\begin{lemma}$($\cite[Theorems 1.3 and 1.8]{MRT20}$)$\label{a lemma from fourier uniformity}
Let $k$ be a given positive integer, and let $5/8<\tau<1$ and $0<\theta<1$ be fixed. Denote by $\mathcal{D}_{k}$ the set of all polynomials in $\mathbb{R}[y]$ of degrees less than $k$. Let $f:\mathbb{N}\rightarrow \mathbb{C}$ be a multiplicative function with $|f(n)|\leq 1$ for any $n\in \mathbb{N}$. Suppose that $X\geq 1$, $X^{\theta}\geq H\geq \exp((\log X)^{\tau})$, and $\eta>0$ are such that
\[
\int_X^{2X} \sup_{p(y)\in \mathcal{D}_{k}}\left| \sum_{x\leq n< x+H} f(n)e(p(n))\right| dx \geq \eta HX.
\]
Then
\[
M(f;AX^{k}/H^{k-1/2}, Q) \ll_{k,\eta,\theta,\tau} 1
\]
for some positive numbers $A,Q\ll_{k,\eta, \theta,\tau} 1.$
\end{lemma}
The following is a well-known result (see, e.g., \cite[(1.12)]{KMT17}) about the ``non-pretentious'' nature of the M\"{o}bius function.
\begin{proposition}\label{Halasz bound}
Let $Q>0$. For $\epsilon>0$, we have for $X$ large enough,
\[
M(\mu;X,Q) \geq (1/3- \epsilon) \log\log X+O(1).
\]
\end{proposition}
By the above proposition and Lemma \ref{a lemma from fourier uniformity}, we have the following result which states that $\mu(n)$ does not correlate with polynomial phases in short intervals on average.
\begin{lemma}\label{the distrbution of moebius in short interval}
Let $k$ be a given positive integer, and let $5/8<\tau<1$ and $0<\theta<1$ be fixed. Suppose that $X\geq 1$ and $X^{\theta}\geq H\geq \exp((\log X)^{\tau})$.  Denote by $\mathcal{D}_{k}$ the set of all polynomials in $\mathbb{R}[y]$ of degrees less than $k$. Then for any $\eta>0$, there is an $X_{0}>0$ (at most depends on $k,\eta, \theta,\tau$) such that whenever $X>X_{0}$,
\[
\int_X^{2X} \sup_{p(y)\in \mathcal{D}_{k}}\Bigg| \sum_{x\leq n< x+H} \mu(n)e\big(p(n)\big)\Bigg| dx<\eta HX.
\]
\end{lemma}
Now applying Lemma \ref{the distrbution of moebius in short interval}, we show the following result.
\begin{theorem}\label{k=1}
Let $\tau\in (5/8,1)$ be given. Let $\{N_{i}\}_{i=0}^{\infty}$ be an increasing sequence of natural numbers with $N_{0}=0$ and $N_{i+1}-N_{i}\geq \exp((\log i)^{\tau})$ for $i$ large enough, and let $\{p_{i}(y)\}_{i=0}^{\infty}$ be a sequence in $\mathbb{R}[y]$ of degrees less than $k$ for some positive integer $k$. Then
\begin{equation}\label{0816formula5}
\lim_{m\rightarrow \infty}\frac{1}{N_{m}}\sum_{i=0}^{m-1}\Bigg|\sum_{N_{i}\leq n<N_{i+1}}\mu(n)e(p_{i}(n))\Bigg|=0.
\end{equation}
\end{theorem}
\begin{proof}
Given $\epsilon$ sufficiently small with $(1-2\epsilon)\tau>5/8$. Choose $h_{m}=\exp((\log N_{m})^{(1-2\epsilon)\tau})$. By Lemma \ref{the distrbution of moebius in short interval} and a dyadic subdivision, for $N_{m}$ large enough,
\begin{equation} \label{eq2}
\sum_{x=0}^{(\lfloor \frac{N_{m}}{h_{m}}\rfloor+1)h_{m}-1}\sup_{p(y)\in \mathcal{D}_{k}}\Bigg|\sum_{x\leq n< x+h_{m}}\mu(n)e(p(n))\Bigg| \leq
\epsilon N_{m} h_{m},
\end{equation}
where $\mathcal{D}_{k}$ is the set of polynomials in $\mathbb{R}[y]$ of degrees less than $k$.

Let $S_{j}=\{lh_{m}+j: l=0,1,\cdot\cdot\cdot, \lfloor\frac{N_{m}}{h_{m}}\rfloor\}$, for $0\leq j\leq h_{m}-1$.  Then, by (\ref{eq2}), there is a $j_{0}$ such that
\begin{equation} \label{eq33}
\sum_{x\in S_{j_{0}}}\sup_{p(y)\in \mathcal{D}_{k}}\Bigg|\sum_{x\leq n< x+h_{m}}\mu(n)e(p(n))\Bigg|\leq \epsilon N_{m}.
\end{equation}
Suppose $S_{j_{0}}\cap [N_{i},N_{i+1})=\{x_{1}^{(i)},x_{2}^{(i)},\cdot\cdot\cdot,x_{l_{i}}^{(i)}\}$, where $x_{1}^{(i)}<x_{2}^{(i)}<\cdot\cdot\cdot<x_{l_{i}}^{(i)}$, $i=0,1,\cdot\cdot\cdot,m-1$. Then \[\Bigg|\sum_{N_{i}\leq n< N_{i+1}}\mu(n)e(p_{i}(n))\Bigg|\leq\sum_{t=1}^{l_{i}-1}\sup_{p(y)\in \mathcal{D}_{k}}\Bigg|\sum_{x_{t}^{(i)}\leq n<x_{t}^{(i)}+h_{m}}\mu(n)e(p(n))\Bigg|+2h_{m}.\] So

\[\sum_{i=0}^{m-1}\Bigg|\sum_{N_{i}\leq n< N_{i+1}}\mu(n)e(p_{i}(n))\Bigg|\leq \sum_{x\in S_{j_{0}}}\sup_{p(y)\in \mathcal{D}_{k}}\Bigg|\sum_{x\leq n< x+h_{m}}\mu(n)e(p(n))\Bigg|+2mh_{m}.\]
By formula (\ref{eq33}),
\[\frac{1}{N_{m}}\sum_{i=0}^{m-1}\Bigg|\sum_{N_{i}\leq n< N_{i+1}}\mu(n)e(p_{i}(n))\Bigg|\leq \frac{\epsilon N_{m}}{N_{m}}+\frac{2mh_{m}}{N_{m}}.\]
Since $N_{i+1}-N_{i}\geq \exp((\log i)^{\tau})$ for $i$ large enough, we have $m\leq N_{m}/\big(\exp((\log N_{m})^{(1-\epsilon)\tau})\big)$ for $N_{m}$ sufficiently large. Inserting this into the above inequality and letting $m\rightarrow \infty$, we obtain \[\limsup_{m\rightarrow \infty}\frac{1}{N_{m}}\sum_{i=0}^{m-1}\Bigg|\sum_{N_{i}\leq n< N_{i+1}}\mu(n)e(p_{i}(n))\Bigg|\leq \epsilon.\] Letting $\epsilon\rightarrow 0$, we obtain equation (\ref{0816formula5}).
\end{proof}
Now we are ready to prove Theorem \ref{k=1rapiddecreasingspeed}, which states the M\"{o}bius disjointness of $e(f(n))$ with the $k$-th difference of $f(n)$ tending to zero as in formula (\ref{basic relation}).
\begin{proof}[Proof of Theorem \ref{k=1rapiddecreasingspeed}]
Firstly, by Proposition \ref{a lemma used in section 5}, for $j\geq k$, there are integers $a_k,...,a_j $ such that for each $n\in \N$,
\begin{align}
\sum_{s=k}^j a_s\cdot \triangle^k f(n+s-k) = f(n+j)-\sum_{l=0}^{k-1} f(n+l)\prod_{t=0,t\neq l}^{k-1} \frac{j-t}{l-t},
\end{align}
where $0\leq a_{s}\leq j^{k-1}$ for $s=k,\ldots,j$. Note that $\sum_{l=0}^{k-1} f(n+l)\prod_{t=0,t\neq l}^{k-1} \frac{j-t}{l-t}=f(n+j)$ for $j=0,\ldots,k-1$.
Then by condition (\ref{basic relation}), for any $j\in \mathbb{N}$,
\[
\Big\|f(n+j)-\sum_{l=0}^{k-1} f(n+l)\prod_{t=0,t\neq l}^{k-1} \frac{j-t}{l-t}\Big\|\leq \frac{Cj^{k}}{\exp((\log n)^{\tau})}.
\]
For any positive integer $M$ with $MC\geq 1$, choose $L_{0}=0$ and $L_{m}=2^{m}\lfloor \exp\big(\log^{\frac{1}{\tau}}(MC2^{mk})\big)+1\rfloor$ for $m=1,2,\ldots$. Then by the above inequality, we have for $n\geq L_{m}$ and $j=0,1,\ldots,2^{m}-1$,
\begin{equation}\label{0817formula8}
\Big\|f(n+j)-\sum_{l=0}^{k-1} f(n+l)\prod_{t=0,t\neq l}^{k-1} \frac{j-t}{l-t}\Big\|\leq \frac{1}{M}.
\end{equation}
Let $d_{m}=(L_{m+1}-L_{m})/2^{m}$. Setting $0=N_{0}<N_{1}<N_{2}<\cdot\cdot\cdot$ with $\{N_{0},N_{1},\ldots\}$ being the set of $\{L_{m}+t2^{m}: m\in \mathbb{N}, 0\leq t\leq d_{m}-1\}$. Assume $L_{m}=N_{k_{m}}$. Then $k_{m+1}=k_{m}+d_{m}$. Note that $k_{0}=0$. Then for $m\geq 1$, \[\exp\big(\log^{\frac{1}{\tau}}(MC2^{mk})\big)\ll k_{m}\ll m\exp\big(\log^{\frac{1}{\tau}}(MC2^{mk})\big).\]
Choose an appropriate $\epsilon>0$ with $\tau-\epsilon>5/8$. Then for $m$ large enough, we have $2^{m}\geq  \exp((\log k_{m+1})^{\tau-\epsilon})$. By the choice of $N_{i}$, this leads to $N_{i+1}-N_{i}>\exp((\log i)^{\tau-\epsilon})$ for $i$ large enough.
Define \[p_{M}(n)=\sum_{l=0}^{k-1}f(N_{i}+l)\prod_{t=0,t\neq l}^{k-1}\frac{n-N_{i}-t}{l-t}\] when $N_{i}\leq n< N_{i+1}$, $i=0,1,\ldots$.
Hence by Theorem \ref{k=1},
\[\lim_{s\rightarrow \infty} \frac{1}{N_{s}}\sum_{i=0}^{s-1}\sum_{N_i\leq n <N_{i+1}}
\mu(n)e(p_{M}(n))= 0,\]
and further by Lemma \ref{subsequence is equivalent all natural number},
\begin{equation}\label{claimed as in this theorem}
\lim_{N\rightarrow \infty} \frac{1}{N}\sum_{n=0}^{N-1}
\mu(n)e(p_{M}(n))= 0.
\end{equation}
Since $\sup_{n\in \mathbb{N}}\|f(n)-p_{M}(n)\|\leq \frac{1}{M}$ by equation (\ref{0817formula8}), $\lim_{M\rightarrow \infty}\sup_{n\in \mathbb{N}}|e(f(n))-e(p_{M}(n))|=0$. Hence it follows from equation (\ref{claimed as in this theorem}) that
\[\lim_{N\rightarrow \infty} \frac{1}{N}\sum_{n=0}^{N-1}
\mu(n)e(f(n))= 0\]
as claimed.
\end{proof}

\section*{Funding}

This work was supported by Grant TRT 0159 from the Templeton Religion Trust. F.W. supported  by the fellowship of China Postdoctoral Science Foundation [2020M670273].

\section*{Acknowledgments}

This research was begun when we visit Harvard University during the 2018-2019 academic year. We thank Arthur Jaffe for support and Liming Ge for valuable discussions. We would also like to thank Jie Wu, Jinxin Xue and Yitwah Cheung for very helpful suggestions and comments on the manuscript.

\appendix
\section{Some properties of the difference operator}
Recall that the difference operator $\triangle$ is defined as $(\triangle f)(n)=f(n+1)-f(n)$ for any arithmetic function $f$. In this section, we give some basic properties of the operator $\triangle^{k}$ for $k\in \mathbb{N}$, which are used in this paper. The following one can be easily deduced by induction.
\begin{proposition}\label{formula for the k-th defference}
For any $k\in \mathbb{N}$ and any arithmetic function $f$, we have
\begin{equation}\label{explicit formula for the k-th defference}
\triangle^{k}f(n)=\sum_{l=0}^{k}(-1)^{k-l}\binom{k}{l}f(n+l).
\end{equation}
\end{proposition}
It is known that $\triangle^{k}f(n)\equiv 0$ if and only if $f(n)$ is a polynomial with degree $k-1$. Moreover, we have the following proposition, which is known as the Lagrange interpolating polynomial.
\begin{proposition}\label{expression with k-th difference zero}
Given integers $J,k$ with $J>k\geq 0$. Suppose that $f(0),\ldots,f(J-1)$ satisfy $\sum_{l=0}^{k}(-1)^{k-l}\binom{k}{l}f(n+l)=0$ for $n=0,\ldots,J-1-k$. Then
\[f(n)=\sum_{j=0}^{k-1}f(j)\prod_{\substack{0\leq i\leq k-1\\i\neq j}}\frac{n-i}{j-i}\]
for $n=0,\ldots,J-1$.
\end{proposition}
\begin{proof}
Suppose $g(n)=\sum_{j=0}^{k-1}f(j)\prod_{\substack{0\leq i\leq k-1\\i\neq j}}\frac{n-i}{j-i}$. Then $g(n)$ is a polynomial of degree at most $k-1$ and satisfies $g(0)=f(0),\ldots,g(k-1)=f(k-1)$. By Proposition \ref{formula for the k-th defference},
\[\triangle^{k}g(n)=\sum_{l=0}^{k}(-1)^{k-l}\binom{k}{l}g(n+l)=0,\:n\geq 0.\]
For any given $x_{0},\ldots,x_{k-1}$, the solution $(x_{k},\ldots,x_{J-1})$ that satisfies $\sum_{l=0}^{k}(-1)^{k-l}\binom{k}{l}x_{n+l}=0$ for $n=0,\ldots,J-k-1$ is unique. Then $(f(k),\ldots,f(J-1))=(g(k),\ldots,g(J-1))$. Hence $f(n)$ is of the form as claimed in this proposition.
\end{proof}
The following simple fact is used in this paper.
\begin{proposition}\label{coefficients are integers}
For any $n\in \mathbb{N}$, $k\geq 1$ and $0\leq j\leq k-1$, $\prod_{\substack{0\leq i\leq k-1\\i\neq j}}\frac{n-i}{j-i}$ is an integer.
\end{proposition}
\begin{proof}
For $0\leq n\leq k-1$ and $n\neq j$, $\prod_{\substack{0\leq i\leq k-1\\i\neq j}}\frac{n-i}{j-i}=0$; for $n=j$, $\prod_{\substack{0\leq i\leq k-1\\i\neq j}}\frac{n-i}{j-i}=1$; for $n\geq k$,
\begin{align*}
 \prod_{\substack{0\leq i\leq k-1\\i\neq j}}\frac{n-i}{j-i}&= \prod_{0\leq i\leq j-1}\frac{n-i}{j-i}\prod_{j+1\leq i\leq k-1}\frac{n-i}{j-i} \\
  &=(-1)^{k-j-1}\frac{n!}{(n-k)!j!(k-1-j)!(n-j)}\\
  &=(-1)^{k-j-1}\binom{n-j-1}{n-k}\binom{n}{n-j}
\end{align*}
is an integer.
\end{proof}\
The next one gives a variant version of Proposition \ref{expression with k-th difference zero}.
\begin{lemma}\label{equivalent conditions}
Given integers $J,k$ with $J>k\geq 0$. Suppose that $x_{0},\ldots, x_{J-1}\in \mathbb{R}$. Then the following two statements are equivalent.
\begin{itemize}
\item[(i)] For $n=0,\ldots,J-1-k$,
\begin{equation}\label{lemma1.151}
\{\sum_{l=0}^{k}(-1)^{k-l}\binom{k}{l}\{x_{n+l}\}\}=0.
\end{equation}

\item[(ii)] For $n=0,\ldots,J-1$,
\begin{equation}\label{lemma1.152}
\{x_{n}\}=\{\sum_{j=0}^{k-1}\{x_{j}\}\prod_{\substack{0\leq i\leq k-1\\i\neq j}}\frac{n-i}{j-i}\},
\end{equation}
where $\{\cdot\}$ denotes the fractional part function.
\end{itemize}
\end{lemma}

\begin{proof}
(i) $\Rightarrow$ (ii). We first show that there are integers $c_{0},c_{1},\ldots,c_{J-1}$ such that $x_{0}+c_{0},x_{1}+c_{1},\ldots,x_{J-1}+c_{J-1}$ satisfy the following linear equations
\begin{equation}\label{equation without fraction part}
\sum_{l=0}^{k}(-1)^{k-l}\binom{k}{l}(x_{n+l}+c_{n+l})=0,\;n=0,\ldots,J-1-k.
\end{equation}
Let $c_{0}=\cdot\cdot\cdot=c_{k-1}=0$. Letting $n=0$ in equation (\ref{equation without fraction part}), by equation (\ref{lemma1.151}) we have that the solution $c_{k}$ is an integer. Repeating the above process with $n=1,\ldots,J-1-k$, we obtain successively solutions $c_{k+1},\ldots,c_{J-1}$, which are all integers. By equation (\ref{equation without fraction part}) and Proposition \ref{expression with k-th difference zero},
\[x_{n}+c_{n}=\sum_{j=0}^{k-1}(x_{j}+c_{j})\prod_{\substack{0\leq i\leq k-1\\i\neq j}}\frac{n-i}{j-i},\; n=0,\ldots, J-1.\]
By Proposition \ref{coefficients are integers}, we have equation (\ref{lemma1.152}).

(ii)$\Rightarrow$ (i). Assume that $x_{0},x_{1},\ldots,x_{J-1}$ satisfy equation (\ref{lemma1.152}). Let
\[g(n)=\sum_{j=0}^{k-1}\{x_{j}\}\prod_{\substack{0\leq i\leq k-1\\i\neq j}}\frac{n-i}{j-i},\;n=0,\dots,J-1.\]
Then $g(n)$ is a polynomial of degree at most $k-1$. By Proposition \ref{formula for the k-th defference},
\[\triangle^{k}g(n)=\sum_{l=0}^{k}(-1)^{k-l}\binom{k}{l}g(n+l)=0,\:n\geq 0.\]
By (ii), $\{g(0)\}=\{x_{0}\},\{g(1)\}=\{x_{1}\},\ldots, \{g(J-1)\}=\{x_{J-1}\}$. Then
\[\{\sum_{l=0}^{k}(-1)^{k-l}\binom{k}{l}\{x_{n+l}\}\}=\{\sum_{l=0}^{k}(-1)^{k-l}\binom{k}{l}\{g(n+l)\}\}=0,\:n=0,\ldots,J-1-k.\]
\end{proof}
The following gives an estimate of $f(n)$ through the initial values and the upper bound of $\triangle^{k}f(n)$.
\begin{proposition}\label{estimate of the value at every point}
Given $n\in \N$ and a real number $c>0$.
Suppose integers $J> k\geq 0$.
If $f(n),f(n+1),\ldots,f(n+J-1)$ satisfy:

(a) $|\triangle^k f(n+j)|\le c$, $j=0,1,..., J-1-k$;

(b) $f(n+j)\in [0,c]$, $j=0,1,...,k-1$.\\
Then we have
\[|f(n+j)|\le (k+1)j^kc,\;j=k,k+1,...,J-1. \]
\end{proposition}
\begin{proof}
By (b), for $j=0,1,...,k-2$,
\[|\triangle f(n+j)|\le c,\]
and by induction,
\begin{align}\label{0619shi1}
|\triangle^m f(n+j)|\le 2^{m-1}c,\; j=0,1,...,k-m-1,\;1\le m\le k-1.
\end{align}
We first claim that, when $1\le m\le k$ and $ k-m\le j\le J-1-m$, we have
\begin{align}\label{0619shi2}
|\triangle^m f(n+j)|\le \sum_{i=0}^{k-m-1} \binom{j-(k-m)+i}{j-(k-m)} 2^{m-1+i}c+\binom{j}{k-m}c,
\end{align}
where
$\binom{0}{0}=1$. In the following we shall prove formula (\ref{0619shi2}).
When $m=k $,
the inequality (\ref{0619shi2}) holds by (a).
Assume inductively that formula (\ref{0619shi2}) holds when $ m=m_0+1\le k $. We shall prove that formula (\ref{0619shi2}) holds when $m=m_{0}$ and $k-m_{0}\leq j\leq J-1-m_{0}$.
For $j=k-m_{0}$, we have
\[|\triangle^{m_{0}} f(n+k-m_{0})|\le | \triangle^{m_{0}+1} f(n+k-m_{0}-1) |+| \triangle^{m_0} f(n+k-m_{0}-1) |. \]
Then by the inductive hypothesis on the $m_{0}+1$ case and formula (\ref{0619shi1}),
\begin{align*}
|\triangle^{m_{0}} f(n+k-m_{0})|\le & 2^{m_{0}-1}c+ \sum_{i=0}^{k-m_{0}-2} \binom{k-m_{0}-1-(k-m_{0}-1)+i}{k-m_{0}-1-(k-m_{0}-1)} 2^{m_{0}+i}c\\
&+\binom{k-m_{0}-1}{k-m_{0}-1}c\\
=& \sum_{i=0}^{k-m_{0}-1} \binom{k-m_{0}-(k-m_{0})+i}{k-m_{0}-(k-m_{0})} 2^{m_{0}-1+i}c+\binom{k-m_{0}}{k-m_{0}}c.
\end{align*}
Note that in the above formula the coefficients before $2^{m_{0}+i}c$ and $2^{m_{0}-1+i}c$ are all $1$.
So formula (\ref{0619shi2}) holds for $m=m_{0}$ and $j=k-m_{0}$.
Now assume inductively that formula (\ref{0619shi2}) holds for $m=m_{0}$ and some $j=j_0\ge k-m_{0}$. When $j=j_0+1$,
we have
\begin{align*}
|\triangle^{m_{0}} f(n+j_{0}+1)|\le& | \triangle^{m_{0}} f(n+j_0) |+| \triangle^{m_{0}+1} f(n+j_0) |\\
\le & \sum_{i=0}^{k-m_{0}-1} \binom{j_0-(k-m_{0})+i}{j_0-(k-m_{0})} 2^{m_{0}-1+i}c+\binom{j_0}{k-m_{0}}c \\
&+ \sum_{i=0}^{k-m_{0}-2} \binom{j_0-(k-m_{0}-1)+i}{j_0-(k-m_{0}-1)} 2^{m_{0}+i}c+\binom{j_0}{k-m_{0}-1}c\\
= & \sum_{i=0}^{k-m_{0}-1} \binom{j_0-(k-m_{0})+i}{j_0-(k-m_{0})} 2^{m_{0}-1+i}c+\binom{j_0}{k-m_{0}}c \\
&+ \sum_{i=1}^{k-m_{0}-1} \binom{j_0-(k-m_{0})+i}{j_0-(k-m_{0}-1)} 2^{m_{0}-1+i}c+\binom{j_0}{k-m_{0}-1}c\\
= & \sum_{i=1}^{k-m_{0}-1} \binom{j_0-(k-m_{0})+i+1}{j_0-(k-m_{0})+1} 2^{m_{0}-1+i}c+2^{m_{0}-1}c+ \binom{j_0+1}{k-m_{0}}c\\
=& \sum_{i=0}^{k-m_{0}-1} \binom{j_{0}+1-(k-m_{0})+i}{j_{0}+1-(k-m_{0})} 2^{m_{0}-1+i}c+ \binom{j_{0}+1}{k-m_{0}}c.
\end{align*}
Then formula (\ref{0619shi2}) holds by the above induction process.
In particular,
taking $m=1$ in formula (\ref{0619shi2}), for $k-1\leq j\leq J-2$,
\begin{align*}
|\triangle f(n+j)|\le \sum_{i=0}^{k-2} \binom{j-k+i+1}{j-k+1} 2^{i}c+\binom{j}{k-1}c.
\end{align*}
Hence for $k\leq j\leq J-1$, by the above inequality,
\begin{align*}
|f(n+j)|\le &|f(n+k-1) |+ \sum_{l=k-1}^{j-1} |\triangle f(n+l)|\\
\le &c+ \sum_{l=k-1}^{j-1} \left(\sum_{i=0}^{k-2} \binom{l-k+i+1}{l-k+1} 2^{i}+\binom{l}{k-1}\right)c\\
= & c+ \sum_{i=0}^{k-2}  \binom{j-1-k+i+2}{i+1}2^ic+ \binom{j}{k}c.
\end{align*}
Since $\binom j k\le j^k$ and $\binom{j-k+i+1}{j-k}2^i\le j^{i+1} $ when $ 0\le i\le k-2$,
\[|f(n+j)|\le (k+1) j^kc,\; k\le j\le J-1. \]
This completes the proof.
\end{proof}
The following proposition shows that $f$ can be approached by polynomials if the $k$-th difference of $f$ is small. This proposition will be used in the proof of Theorem \ref{k=1rapiddecreasingspeed} in Section 5.
\begin{proposition}\label{a lemma used in section 5}
Given $j\geq k\geq 1$.
There are constants $a_k,...,a_j $ such that for each arithmetic function $f$ and each $n\in \N$,
\begin{align}\label{0822shi1}
\sum_{l=k}^j a_l\cdot \triangle^k f(n+l-k) = f(n+j)-\sum_{m=0}^{k-1} f(n+m)\prod_{i=0,i\neq m}^{k-1} \frac{j-i}{m-i}.
\end{align}
\end{proposition}
\begin{proof}
Let
$$g(l)=\binom{j-l+k-1}{k-1}=\frac{(j-l+1)(j-l+2)\cdots (j-l+k-1) }{(k-1)!}$$
be a polynomial of $l$ with degree $k-1$.
Choose $a_l =g(l)$ when $k\le l\le j+k-1$.
So $a_l=0$ when $l= j+1,...,j+k-1$.
By equation (\ref{explicit formula for the k-th defference}),
the left side of equation (\ref{0822shi1}) is
\begin{align}\label{0822shi2}
\sum_{l=k}^j a_l\sum_{t=0}^k (-1)^{k-t} \binom{k}{t} f(n+l-k+t).
\end{align}
In the following, we consider the coefficients of $f(n+m)$ in the above formula for $m=0,\ldots,j$. When $m=j$, the coefficient of $f(n+m)$ in (\ref{0822shi2}) is $a_j=1$. For the case $j\geq 2k$, when $k\le m\le j-k$,
the coefficient of $f(n+m)$ in (\ref{0822shi2}) is
\[\sum_{t=0}^k a_{m-t+k}(-1)^{k-t} \binom k t = (-1)^k \sum_{t=0}^k g(m+t)(-1)^{k-t} \binom k t=(-1)^k \triangle^k g(m )=0.\]
Notice that $a_l=0$ for $l= j+1,...,j+k-1$.
When $j-k< m\le j-1$, the coefficient of $f(n+m)$ in (\ref{0822shi2}) is
\begin{equation}\label{formula0824}
\sum_{t=m+k-j}^k a_{m-t+k}(-1)^{k-t} \binom k t = \sum_{t=0}^k a_{m-t+k}(-1)^{k-t} \binom k t=(-1)^k \triangle^k g(m )=0.
\end{equation}
For the case $j\leq 2k-1$, when $k\leq m\leq j-1$, by a similar argument to equation (\ref{formula0824}), we have that the coefficient of $f(n+m)$ in (\ref{0822shi2}) is $0$.
Hence there are constants $c_0,...,c_{k-1}$ such that
\begin{align}\label{0822shi3}
\sum_{l=k}^j a_l\cdot \triangle^k f(n+l-k) = f(n+j)-\sum_{m=0}^{k-1} c_m f(n+m)
\end{align}
holds for each $n\in \mathbb{N}$.
To compute $c_m$,
let $f_m$ be the polynomial of degree $k-1$ such that $f_m(m)=1$, $f_m(t)=0$ for $0\le t\le k-1$ and $t\neq m$.
So
\[f_m(t)= \prod_{i=0,i\neq m}^{k-1}\frac{t-i}{m-i}. \]
Note that $\triangle^k f_m = 0$ and equation (\ref{0822shi3}) holds for any arithmetic function $f$. Let $f=f_{m}$ and $n=0$ in equation (\ref{0822shi3}), then
\[c_m=f_m(j)=\prod_{i=0,i\neq m}^{k-1}\frac{j-i}{m-i}.\]
Hence equation (\ref{0822shi1}) holds with $a_{l}=g(l)$ for $l=k,\ldots,j$.
\end{proof}
\section{A lemma}
In this section, we prove the following lemma which is used in the proof of Theorem \ref{SMDC implies the distribution of moebius in short interval}. There are some other methods to prove the following result, while for self-containing, we provide a proof that is adapt to our situation.
\begin{lemma}\label{disjointness and distribution in short interval}
Let $w(n)$ be a bounded arithmetic function and $\mathcal{D}$ be a non-empty family consisting of real-valued arithmetic functions. Then the following two conditions are equivalent.
\begin{itemize}
\item[(i)] For any increasing sequence $ \{N_i\}_{i=0}^\infty$ of natural numbers with $N_0=0$ and $\lim_{i\rightarrow\infty} (N_{i+1}-N_i)=\infty$, and any sequence $\{f_{i}(n)\}_{i=0}^{\infty}$ in $\mathcal{D}$, we have
\begin{equation}\label{in arithmetic progression}
\lim_{m\rightarrow \infty}\frac{1}{N_{m}}\sum_{i=0}^{m-1}\Bigg|\sum_{{N_{i}}\leq n<N_{i+1}}w(n)e(f_{i}(n))\Bigg|=0.
\end{equation}

\item[(ii)] We have
\begin{equation}\label{eq5}
\lim_{h\rightarrow \infty}\limsup_{X\rightarrow \infty}\frac{1}{Xh}\int_{X}^{2X} \sup_{f\in\mathcal{D}}\Bigg|\sum_{x\leq n<x+h}w(n)e(f(n))\Bigg|dx=0.
\end{equation}
\end{itemize}
\end{lemma}
\begin{proof}
We first prove that $ $(i)$\Rightarrow $(ii)$ $. Assume on the contrary that formula (\ref{eq5}) does not hold. Then there is a $\delta\in (0,1)$ and a sequence $\{h_{j}\}_{j=0}^{\infty}$ of positive integers with $\delta h_{j}>4$ and $\lim_{j\rightarrow\infty}h_{j}=\infty$, such that
\[\limsup_{X\rightarrow \infty}\frac{1}{X}\int_{X}^{2X} \sup_{f\in \mathcal{D}}\Bigg|\sum_{x\leq n<x+h_{j}}w(n)e(f(n))\Bigg|dx>2\delta h_{j}.\]
Given $h_{j}$, choose $X_{j}$ large enough with $\delta X_{j}>16h_{j}$ and $X_{j}>4X_{j-1}$ satisfying
\[\int_{X_{j}}^{2X_{j}} \sup_{f\in \mathcal{D}}\Bigg|\sum_{x\leq n<x+h_{j}}w(n)e(f(n))\Bigg|dx>2\delta X_{j}h_{j}.\]
By the pigeonhole principal, there is a $y_{j}\in [0, h_{j})$, such that
\[\sum_{l_{j}=\lfloor X_{j}/h_{j}\rfloor-1}^{\lfloor 2X_{j}/h_{j}\rfloor+1}\bigg(\sup_{f\in \mathcal{D}}\Bigg|\sum_{n=l_{j}h_{j}+y_{j}}^ {(l_{j}+1)h_{j}+y_{j}-1}w(n)e(f(n))\Bigg|\bigg)>\frac{3}{2}\delta X_{j}.\]
Furthermore, for each $l_{j}$ with $\lfloor X_{j}/h_{j}\rfloor-1 \leq l_{j}\leq \lfloor2X_{j}/h_{j}\rfloor+1$, we can find $g_{l_{j}}(n)\in \mathcal{D}$ such that

\begin{equation}\label{eq6}
\sum_{l_{j}=\lfloor X_{j}/h_{j}\rfloor-1}^{\lfloor 2X_{j}/h_{j}\rfloor+1}\Bigg|\sum_{n=l_{j}h_{j}+y_{j}}^ {(l_{j}+1)h_{j}+y_{j}-1}w(n)e(g_{l_{j}}(n))\Bigg|>\delta X_{j}.
\end{equation}
Now we construct $\{N_{i}\}_{i=0}^{\infty}$ and $\{f_{i}(n)\}_{i=0}^{\infty}$ in the following way. Choose $N_{0}=0$ and $N_{i}= \lfloor X_{J+1}/h_{J+1}\rfloor h_{J+1}+th_{J+1}+y_{J+1}$ when $i=\sum_{j=1}^{J}\lfloor X_{j}/h_{j}\rfloor-J+1+t$, $t=0,1,\ldots, \lfloor X_{J+1}/h_{J+1}\rfloor-2$, where $J=0,1,\ldots$. Then $\lim_{i\rightarrow \infty}(N_{i+1}-N_{i})=\infty$. For $J=0,1,\ldots$, we choose $f_{i}(n)=g_{\lfloor X_{J+1}/h_{J+1}\rfloor+t}(n)$ when $i=\sum_{j=1}^{J}\lfloor X_{j}/h_{j}\rfloor-J+1+t$, $t=0,1,\ldots, \lfloor X_{J+1}/h_{J+1}\rfloor-3$, and $f_{i}(n)=g_{\lfloor X_{1}/h_{1}\rfloor}(n)$ otherwise.
Then by equation (\ref{in arithmetic progression}), there is an $m_{0}$ and $J_{0}$ with $m_{0}=\sum_{j=1}^{J_{0}-1}\lfloor X_{j}/h_{j}\rfloor-J_{0}+2$, such that
\begin{equation}\label{eq7}
\frac{1}{N_{m_{0}
+\lfloor X_{J_{0}}/h_{J_{0}}\rfloor-2}}\Bigg(\sum_{i=0}^{m_{0}
+\lfloor X_{J_{0}}/h_{J_{0}}\rfloor-3}\Bigg|\sum_{N_{i}\leq n <N_{i+1}}w(n)e(f_{i}(n))\Bigg|\Bigg)<\frac{\delta}{4}.
\end{equation}
Note that $N_{m_{0}}=\lfloor X_{J_{0}}/h_{J_{0}}\rfloor h_{J_{0}}+y_{J_{0}}$ and $N_{m_{0}
+\lfloor X_{J_{0}}/h_{J_{0}}\rfloor-2}=2(\lfloor X_{J_{0}}/h_{J_{0}}\rfloor-1)h_{J_{0}}+y_{J_{0}}$.
By formula (\ref{eq6}), we have that $\sum_{i=m_{0}}^{m_{0}
+\lfloor X_{J_{0}}/h_{J_{0}}\rfloor-3}|\sum_{N_{i}\leq n <N_{i+1}}w(n)e(f_{i}(n))|>\frac{\delta X_{J_{0}}}{2}$. Thus the left side of formula (\ref{eq7})$>\frac{(\delta/2) X_{J_{0}}}{2X_{J_{0}}-h_{J_{0}}}>\frac{\delta}{4}$. This contradicts the right side of formula $(\ref{eq7})$. Hence equation (\ref{eq5}) holds.

Next, we show that $ $(ii)$ \Rightarrow $(i)$ $.
Given $\epsilon$ with $0<\epsilon<1$. Let $h$ be a fixed sufficiently large positive integer. By equation (\ref{eq5}) and a dyadic subdivision, there is an $m_{0}\geq 1$ such that whenever $m>m_{0}$, we have  $N_{m}-N_{m-1}>\frac{2}{\epsilon}h$ and
\begin{equation} \label{eq0904}
\sum_{x=0}^{(\lfloor N_{m}/h\rfloor+1)h-1}\sup_{f\in \mathcal{D}}\Bigg|\sum_{x\leq n< x+h}w(n)e(f(n))\Bigg| \leq
\epsilon N_{m}h.
\end{equation}
Let $S_{j}=\{lh+j: l=0,1,\cdot\cdot\cdot, \lfloor N_{m}/h\rfloor\}$, for $0\leq j\leq h-1$.  Then, by (\ref{eq0904}), there is a $j_{0}$ such that
\begin{equation} \label{0904formula2}
\sum_{x\in S_{j_{0}}}\sup_{f\in \mathcal{D}}\Bigg|\sum_{x\leq n< x+h}w(n)e(f(n))\Bigg| \leq \epsilon N_{m}.
\end{equation}
Suppose that $S_{j_{0}}\cap [N_{i},N_{i+1})=\{x_{1}^{(i)},x_{2}^{(i)},\cdot\cdot\cdot,x_{l_{i}}^{(i)}\}$, where $x_{1}^{(i)}<x_{2}^{(i)}<\cdot\cdot\cdot<x_{l_{i}}^{(i)}$, $i=0,1,\cdot\cdot\cdot,m-1$. Then
\[\Bigg|\sum_{N_{i}\leq n<N_{i+1}}w(n)e(f_{i}(n))\Bigg|\leq\sum_{t=1}^{l_{i}-1}\sup_{f\in\mathcal{D}}\Bigg|\sum_{x_{t}^{(i)}\leq n<x_{t}^{(i)}+h}w(n)e(f(n))\Bigg|+2h.\]
Hence
\[\sum_{i=0}^{m-1}\Bigg|\sum_{N_{i}\leq n<N_{i+1}}w(n)e(f_{i}(n))\Bigg|\leq \sum_{x\in S_{j_{0}}}\sup_{f\in \mathcal{D}}\Bigg|\sum_{x\leq n< x+h}w(n)e(f(n))\Bigg| +2mh.\]
By formula (\ref{0904formula2}),
\[\frac{1}{N_{m}}\sum_{i=0}^{m-1}\Bigg|\sum_{N_{i}\leq n<N_{i+1}}w(n)e(f_{i}(n))\Bigg|\leq \frac{\epsilon N_{m}}{N_{m}}+\frac{2mh}{N_{m}}<2\epsilon.\]
So we obtain equation (\ref{in arithmetic progression}).
\end{proof}

\end{document}